\documentclass[]{scrartcl}
\usepackage[english]{babel}

\usepackage{amsmath, amsthm, amsfonts, amssymb}
\usepackage{bbm} 

\usepackage{graphicx}
\usepackage{xcolor}
\usepackage{tikz}
\usepackage{pgfplots}
\usepackage{todonotes}
\usetikzlibrary{arrows.meta}
\pgfplotsset{compat=1.18}
\usepackage{verbatim}
\usepackage{}
\usepackage{mathptmx}
\newtheorem{thm}{Theorem}

\newtheorem{obs}[thm]{Observation}
\newtheorem{definition}[thm]{Definition}

\newcommand{\dw}{(\omega)\,}
\newcommand{\orb}{\operatorname{orb}}
\newcommand{\Inv}{\operatorname{Inv}}

\newcommand{\diam}{\operatorname{diam}}
\newcommand{\dist}{\operatorname{dist}}

\definecolor{plotblue2}{RGB}{0,0,128}
\definecolor{plotred}{RGB}{220,20,60}
\definecolor{plotgreen}{RGB}{34,120,60}
\usepackage{hyperref}

\KOMAoptions{abstract=true}

\title{Sharkovskiĭ's theorem under small random perturbations}

\author{
  Isabella Alvarenga\thanks{Institut Denis Poisson, Université d'Orléans, Orléans, 45100, France. 
  Email: \texttt{isabella.goncalves-de-alvarenga@cnrs.fr}} 
  \and 
  Daniel Miranda Machado\thanks{Centro de Matemática, Computação e Cognição, Universidade Federal do ABC, Santo André, SP, 09210-580, Brazil. 
  Email: \texttt{daniel.miranda@ufabc.edu.br}}
}

\date{}

\begin{document}

\maketitle

\begin{abstract}
We establish a Sharkovskiĭ-type theorem for a class of discrete random dynamical systems via the random Conley index. Using the continuation property of the Conley index, we extend classical forcing results to random systems obtained from small random perturbations of one-dimensional maps. In contrast to earlier measure-theoretic results, which are typically subject to an inherent period-doubling ambiguity (realizing period $n$ or $2n$), our topological approach allows us to detect random periodic points and orbits with precise minimal periods. This yields realisation results for arbitrary finite tails of the Sharkovskiĭ ordering. These results are illustrated by constructing random periodic orbits for perturbed versions of the tent map and the logistic map.
\end{abstract}

\setkomafont{descriptionlabel}{\normalfont\bfseries}
\begin{description}
    \item[Keywords:] Random dynamical systems, Sharkovskiĭ's theorem, Conley index, Random periodic orbits, Interval maps
    \item[MSC Classification:] 37H10, 37E05, 37E15, 47H04, 47H40
\end{description}

\raggedbottom

\section{Introduction} \label{intro}

In the study of discrete dynamical systems, periodic orbits play a central role. A fundamental result describing the possible periods of such orbits on an interval is Sharkovskiĭ's theorem \cite{sharkovskiui1995coexistence}, which we now recall.

Let~$C\subset \mathbb R$ be a compact interval. For a given function~$f:C\rightarrow C$, we associate to it the dynamical system  given by the iterations of~$f$. We say that~$f$ admits a periodic orbit of period~$k\in\mathbb N$ if its dynamical system admits a periodic orbit of period~$k$, i.e., if there exists a point~$p\in\mathbb R$ in the domain of~$f$ such that
\[f^k(p)=p\quad \text{and}\quad f^j(p)\neq p \text{ for all } j\in\{1,\dots,k-1\}.\] 

Consider over $\mathbb{N}$ the following ordering:
\[
\begin{array}{c}
3 \prec 5 \prec 7 \prec 9 \prec \dots \prec
2\cdot 3 \prec 2\cdot 5 \prec 2\cdot 7 \prec \dots \prec
2^2\cdot 3 \prec 2^2\cdot 5 \prec 2^2\cdot 7 \prec \dots \\
\prec 2^3\cdot 3 \prec 2^3\cdot 5 \prec \dots \prec
\cdots \prec 2^3 \prec 2^2 \prec 2 \prec 1.
\end{array}
\]

\begin{thm}[Sharkovskiĭ's Forcing Theorem \cite{sharkovskiui1995coexistence}]
\label{thm:shark} Let~$f:C\rightarrow C$ be a continuous function. Assume that~$f$ admits a periodic orbit of period~$p\in\mathbb{N}$. If~$p\prec q$, then~$f$ also admits a periodic orbit of period~$q$.  
\end{thm}


In the original work   \cite{sharkovskiui1995coexistence}, more is shown. For~$n\in\mathbb N$ arbitrary but fixed, let~$\mathcal T (n)$ be  the Sharkovskiĭ tail given by
\begin{equation}
    \mathcal T (n):=\{m\in\mathbb N\,:\,m=n \text{ or }n\prec m\}
\end{equation}

\begin{thm}[Sharkovskiĭ's Realisation Theorem \cite{sharkovskiui1995coexistence}]\label{sharkRealisation}
    Let~$n\in \mathbb N$ be arbitrary but fixed. There exists~$f:[0,1]\rightarrow[0,1]$ such that~$f$ has periodic points of all periods~$m\in\mathcal T(n)$, and no periodic points of any period~$m\notin \mathcal T (n)$. 
\end{thm}


In the deterministic setting, Sharkovskiĭ’s theorem and its many refinements are classical; see, for instance, \cite{sharkovskiui1995coexistence,LiYorke1975} and the references therein. Extensions to higher-dimensional settings, such as triangular maps on \(\mathbb{R}^n\), and to graph-like structures (trees and graphs) can be found in \cite{Kloeden1979} and \cite{Baldwin1991,AlsedaLlibreMisiurewicz2000}, respectively. 

In the random setting, Sharkovskiĭ’s theorem has been extended along different directions. An early systematic treatment in the framework of random dynamical systems is due to Klünger~\cite{klunger2001periodicity}, who introduced several notions of random periodicity and proved Sharkovskiĭ-type forcing theorems for one-dimensional random dynamical systems under suitable assumptions on the driving metric dynamical system. A complementary randomization approach was developed by Andres~\cite{andres2008randomization}, who formulated an abstract scheme to transfer classical Sharkovskiĭ-type results to random systems via a transformation to a deterministic setting. In a more functional-analytic direction, Barbarski~\cite{Barbarski2011} extended Sharkovskiĭ’s theorem to spaces of measurable functions and \(L^p\)-spaces, while Andres and Barbarski~\cite{AndresBarbarski2016} obtained randomized Sharkovskiĭ-type results for multivalued maps and scalar differential inclusions, proving coexistence of random subharmonic solutions of all orders. More recently, Andres~\cite{Andres2019} applied these randomized Sharkovskiĭ-type theorems to random impulsive differential equations and inclusions, establishing coexistence of random periodic solutions with periods forced by the Sharkovskiĭ ordering. The present work contributes to this line of research by establishing realisation results for arbitrary finite Sharkovskiĭ tails via random periodic points and random periodic orbits in the framework of random dynamical systems.

Working in the framework of random dynamical systems on a compact interval, we study how Sharkovskiĭ's ordering behaves under small random perturbations of a deterministic map $f\in\mathcal H_C$.
Our first main result is a forcing theorem for \emph{finite} Sharkovskiĭ tails: if $f$ has a hyperbolic periodic orbit of each period in a prescribed finite tail, then every sufficiently small random perturbation admits random periodic objects of exactly those periods.
We then show that the same periods can be realized in three complementary senses: as random periodic points, as set-valued random periodic orbits, and in a geometric $(\delta,k)$--sense that controls spatial localisation along the cycle.
Finally, we illustrate the abstract results by explicit constructions for random perturbations of the tent map and the logistic map.

 This paper is organized as follows.
Section~\ref{MainResults} introduces the random dynamical system (RDS) framework and the notions of random periodicity used throughout, and states our main theorems.
Section~\ref{DefConleyDetAndRand} reviews the deterministic and random Conley index in the form needed here, and Section~\ref{ProofMainThm} contains the proofs of the main results.
In Section~\ref{Examples} we present two illustrative examples, showing how our theorems produce random periodic orbits for perturbations of the tent map and the logistic map.

\section{Main results} \label{MainResults}

\subsection{Preliminaries on random dynamical systems}\label{preliminariesonRDS}
Random dynamical systems describe the evolution of a system under the combined action of deterministic dynamics and a driving noise; see \cite{arnold1995random}. In our setting, the noise is modelled by a metric dynamical system on a probability space, and the dynamics on the phase space are given by a cocycle over this base.

Throughout the rest of this work, we fix~$(\Omega, \mathcal{F},\mathbb{P})$ a probability space and~$(X, d)$ a metric space. We also consider~$X$ as a measurable space by equipping it with the Borel $\sigma$-algebra $\mathcal{B}(X)$.

Let~$t\in\mathbb T$ be a time parameter, where~$\mathbb T$ is $\mathbb N$ or $\mathbb Z$. Consider~$(\theta_t)_{t\in \mathbb{T}}$ a flow defining a metric dynamical system over~$(\Omega, \mathcal{F}, \mathbb{P})$ such that the function~$(t,\omega)\mapsto \theta_t\omega$ from $\mathbb{T}\times \Omega$ to $\Omega$ is measurable and the measure~$\mathbb{P}$ invariant under~$\theta$, i.e.,~$\theta_t\mathbb{P}=\mathbb{P}$. A \textbf{random dynamical system} (RDS)~$\phi$ is a function~$\phi:\mathbb{T}\times \Omega \times X \rightarrow X$ which is~$(\mathcal{B}(\mathbb T)\otimes\mathcal F\otimes\mathcal B(X),\,\mathcal B(X))$-measurable and it satisfies the cocycle property:
\begin{equation}
    \phi(0,\omega)=\text{id}_X, \qquad  \label{posCocycleComp}
                \phi(t+s,\omega) = \phi(t,\theta_s\omega)\circ\phi(s,\omega) \text{ for all }t,s\in \mathbb{T} \text{ and }\omega \in \Omega.
\end{equation}

To state Sharkovskiĭ's original result, we considered a class of dynamical systems generated by the iterations of a map~$f$. 
We proceed in a similar manner for the random case, and thus we now describe a more specific class of random dynamical systems, namely those generated by the iterations of a random map. 

To this end, we consider a \textbf{random map}~$\varphi$, defined as a function  \(\varphi:\Omega\times X\to X\) such that
for each \(\omega\in\Omega\), \(\varphi(\omega,\cdot)\) is continuous, and
\((\omega,x)\mapsto\varphi(\omega,x)\) is \((\mathcal F\otimes\mathcal B(X),\mathcal B(X))\)-measurable.
Given a random map \(\varphi\), we define the associated RDS \(\phi\) by
\begin{equation}\label{eq_RDS_GenByMap}
  \phi(0,\omega,x)=x,\qquad
  \phi(n,\omega,x)
    := \varphi(\theta_{n-1}\omega)\circ\cdots\circ\varphi(\omega)(x),
    \quad n\ge1.
\end{equation}

It can be shown that this construction satisfies the properties required for a random dynamical system when~$\theta:\mathbb{T}\times\Omega \rightarrow\Omega$ is a measurable fixed flow.

Next, we introduce a bit of notation. Let $\varphi$ be a random map and let $\phi$
be the random dynamical system generated by $\varphi$ as in
\eqref{eq_RDS_GenByMap}. For each $k\in\mathbb{N}$ we define the
time-$k$ random map
\[
  \varphi^{k}(\omega,\cdot) := \phi(k,\omega,\cdot), \qquad \omega\in\Omega.
\]

In other words, $\varphi^{k}(\omega)\colon C\to C$ denotes the $k$–step map
of the cocycle at base point $\omega$.

For the remainder of the paper, unless stated otherwise, we consider discrete time ($\mathbb{T} = \mathbb{N}$).

\subsection{Statement of main theorems} \label{StatementMainThm}

In this subsection, we present the main result: a version of Sharkovskiĭ's theorem for small random perturbations of a dynamical system generated by the iterations of a continuous map. To state this result rigorously, we first introduce the notion of a random analogue of a periodic point.

A detailed discussion of this concept is provided in the next subsection. For now, we present three definitions that capture this analogue. The first two have appeared previously in the literature \cite{klunger2001periodicity}, whereas the third is a novel definition proposed in this work. We emphasize that this new definition incorporates a geometric aspect that is absent in the other two, which distinguishes our approach from previous work. In what follows, we fix a random dynamical system~$\phi$ generated by~$\varphi$  with a measurable driving flow \(\theta=(\theta_t)_{t\in\mathbb T}\) on \((\Omega,\mathcal F,\mathbb P)\).

\begin{definition}[~\cite{klunger2001periodicity}] \label{def_randomPeriodicPoint}
    A random variable~$x : \Omega \to \mathbb{R}$ is called a \textbf{random periodic point of period}~$k \in \mathbb{N}$ if
\[
\phi(k, \omega)\, x(\omega) = x(\theta_k \omega) \quad \mathbb{P}\text{-almost surely}.
\]
 
It is said to have \textbf{minimal random period}~$k$ if, in addition
\[
\mathbb{P}\Big(\phi(l, \omega)\, x(\omega) \neq x(\theta_l \omega)\Big) > 0\quad \text{ for every }l \in \{1, \dots, k-1\}.
\]
\end{definition}

A \textbf{random set}~$P$ is a function from~$\Omega$ to~$\mathcal P (X)$ where~$\mathcal P(X)$ is the power set of~$X$. It is said to be \textbf{random invariant} if\[\varphi\dw P(\omega)=P(\theta\omega)\quad \mathbb{P}\text{-a.s.}\] 

For a given set~$A$, we let~$\#A$ denote the cardinality of~$A$. 

\begin{definition}[~\cite{klunger2001periodicity}]\label{def_randomPeriodicOrbit}
We say that a random invariant set~$P$ is a \textbf{random periodic orbit of period}~$k$ if~$\#P\dw=k$ $\mathbb{P}$-a.s. We say that~$P$ has \textbf{minimal period}~$k$ if~$P$ does not contain any random invariant proper and non-empty subset.
\end{definition}

Let $x : \Omega \to X$ be a random variable and~$k\in\mathbb N$.  
For a given $\omega \in \Omega$ and~$l \in \{0,1,\dots,k-1\}$, define the set
\[
S_l(\omega) := \big\{ \phi(mk + l, \omega)\, x(\omega) \,\colon m \in \mathbb{N}_0 \big\}.
\]

\begin{definition}[$(\delta,k)$--random periodic orbit]\label{def:epsilonkorbitaale} 
Let $\delta>0$ and $k\in\mathbb{N}$.
A random variable $x:\Omega\to X$ is called a \emph{$(\delta,k)$--random periodic point} if the sets~$S_\ell(\omega)$ satisfy the following properties for $\mathbb{P}$-almost every $\omega\in\Omega$:
\begin{enumerate}
  \item Each fibre $S_\ell(\omega)$ is $\delta$--small, in the sense that
    \[
      \operatorname{diam}\bigl(S_\ell(\omega)\bigr)
        := \sup\{d(y,z):y,z\in S_\ell(\omega)\} < \delta.
    \]
  \item Different fibres are $\delta$--separated:
    \[
      \operatorname{dist}\bigl(S_i(\omega),S_j(\omega)\bigr)
        := \inf\{d(y,z):y\in S_i(\omega),\,z\in S_j(\omega)\} > \delta
        \quad \text{whenever } i\neq j.
    \]
\end{enumerate}
The associated \emph{$(\delta,k)$--random periodic orbit} is the random family $\{S_\ell(\omega)\}_{\ell=0}^{k-1}$.
\end{definition}

After introducing the random analogues of periodic points, we present two additional
definitions that will be useful for characterizing small random perturbations of a
deterministic dynamical system. Throughout the rest of this this subsection we fix a compact interval
$C\subset\mathbb{R}$ and a continuous map $f:C\to C$.

\begin{definition}[Finite Sharkovskiĭ tail]
  A non-empty finite set $\mathcal{T}\subset \mathbb{N}$ is called a
  \emph{finite Sharkovskiĭ tail} if there exists $n\in\mathbb{N}$ such that:
  \begin{enumerate}
    \item $n\in \mathcal{T}$ and $\mathcal{T}\subset \mathcal{T}(n)$;
    \item $\mathcal{T}$ is an initial segment of $\mathcal{T}(n)$ with
      respect to the order $\succ $, i.e.\ whenever
      $n \succ  k \succ  \ell$ and $\ell\in\mathcal{T}$, then
      $k\in\mathcal{T}$.
  \end{enumerate}
  Equivalently, there exist distinct integers
  $k_1,\dots,k_m\in\mathbb{N}$ such that
  \[
    n = k_1 \succ  k_2 \succ  \cdots \succ  k_m
    \quad\text{and}\quad
    \mathcal{T} = \{k_1,\dots,k_m\}.
  \]
\end{definition}

Let $\{p_0,\dots,p_{k-1}\}\subset \mathrm{int}\,C$ be a periodic orbit of $f$ of period
  $k\in\mathbb{N}$, where $p_j = f^j(p_0)$ for $j=0,\dots,k-1$. We say that this orbit is
  \emph{hyperbolic} if there exists a neighbourhood $U$ of the orbit on which $f$ is $C^1$ and
  \[
    \bigl|(f^k)'(p_j)\bigr|\neq 1
    \qquad \text{for all } j=0,\dots,k-1.
  \]

\begin{definition}\label{def:whatever-HC-is}
Let $\mathcal T\subset\mathbb N$ be a finite Sharkovskiĭ tail. We denote by
$\mathcal H_{\mathcal T}$ the family of continuous maps $f:C\to C$ such that,
for every $k\in\mathcal T$, every periodic orbit of $f$ contained in $C$ with
minimal period $k$ is hyperbolic and contained in $\operatorname{int}C$.
\end{definition}

\begin{definition}\label{def:R1eps-simple}
Let $C\subset\mathbb{R}$ be a compact interval and $f\in\mathcal{H}_{\mathcal{T}}$.
Fix a finite Sharkovskiĭ tail $\mathcal{T}$ and let
\[
  S_k=\{p_{k,1},\dots,p_{k,k}\}\subset\mathrm{int}\,C
\]
be a hyperbolic periodic orbit of $f$ of minimal period $k$ for each
$k\in\mathcal{T}$. Choose pairwise disjoint closed intervals
\[
  N_{k,j}\subset\mathrm{int}\,C,\qquad p_{k,j}\in\mathrm{int}(N_{k,j}),
\]
and set
\[
  N:=\bigcup_{k\in\mathcal{T}}\bigcup_{j=1}^{k} N_{k,j}.
\]

Given $\epsilon>0$, we denote by $\mathcal{R}^1_\epsilon(f)$ the class of
random maps $\varphi:\Omega\times C\to C$ such that:

\begin{enumerate}
  \item For each $\omega\in\Omega$, the map $x\mapsto\varphi(\omega,x)$ is
  continuous on $C$ and piecewise \(C^1\), with all possible
  breakpoints contained in $C\setminus N$.

  \item On the neighbourhood $N$ the maps are uniformly \(C^1\)–close to $f$:
  \[
    \sup_{\omega\in\Omega}\sup_{x\in N}
      \Bigl(|\varphi(\omega,x)-f(x)|
            +|\partial_x\varphi(\omega,x)-f'(x)|\Bigr)<\epsilon.
  \]

  \item On the whole interval $C$ the maps are uniformly \(C^0\)–close to $f$:
  \[
    \sup_{\omega\in\Omega}\sup_{x\in C}
      |\varphi(\omega,x)-f(x)|<\epsilon.
  \]
\end{enumerate}
\end{definition}

We finally state the two main results proved in this work. They are the random analogues
of Theorems~\ref{thm:shark} and~\ref{sharkRealisation}, respectively.

\begin{thm}[Random Sharkovskiĭ forcing for finite tails]
\label{thm:main_forcing}
Let $\mathcal{T}$ be a finite Sharkovskiĭ tail and  $f\in \mathcal{H}_{\mathcal{T}}$.
Denote by $n$ the first element of $\mathcal{T}$ in Sharkovskiĭ's ordering
(i.e.\ $\mathcal{T}\subset \mathcal{T}(n)$ and $n$ is maximal in $\mathcal{T}$
with respect to $\succ$).

Then there exists $\epsilon>0$ such that, for every RDS $\phi$ generated by an
arbitrary $\varphi \in \mathcal{R}^1_{\epsilon}(f)$ with noise $\theta$ ergodic
over $\mathbb{P}$, the following properties hold:
\begin{enumerate}
  \item $\phi$ admits a random periodic point of minimal period $k$ for every
    $k\in \mathcal{T}$ in the sense of Definition~\ref{def_randomPeriodicPoint};
  \item if, in addition, $\theta^k$ is ergodic for all $k\in \mathbb{N}$,
    then $\phi$ admits a random periodic orbit of minimal period $k$ for
    every $k\in \mathcal{T}$ in the sense of
    Definition~\ref{def_randomPeriodicOrbit};
  \item there exists $\delta>0$ such that $\phi$ admits a $(\delta, k)$--random
    periodic orbit for every $k\in \mathcal{T}$ in the sense of
    Definition~\ref{def:epsilonkorbitaale}.
\end{enumerate}
\end{thm}

\begin{thm}[Random realisation of finite Sharkovskii tails]
\label{thm:SharkRealisationAle} 
Let $\mathcal{T}\subset\mathbb{N}$ be a finite Sharkovskii tail. Then the
following hold:

\begin{enumerate}
  \item There exists a non-trivial RDS $\phi$ with noise
    $\theta$ ergodic over $\mathbb{P}$ such that the set of minimal
    periods of random periodic points of $\phi$ contains $\mathcal{T}$.

  \item There exists a non-trivial RDS $\phi$ with noise $\theta$ such
    that $\theta^n$ is ergodic over $\mathbb{P}$ for every $n\in\mathbb{N}$
    and the set of minimal periods of random periodic orbits of $\phi$
    contains $\mathcal{T}$.

  \item There exist $\delta>0$ and a non-trivial RDS $\phi$ with noise
    $\theta$ ergodic over $\mathbb{P}$ such that the set of minimal
    periods of $(\delta,k)$--random periodic orbits of $\phi$ contains
    $\mathcal{T}$.
\end{enumerate}
\end{thm}

\subsection{Connections with other works} \label{Connections}

In this subsection we place our main theorem in the context of existing random analogues of Sharkovskiĭ’s theory.

We first compare Theorem~\ref{thm:main_forcing} with the randomization scheme of Andres~\cite{andres2008randomization}.
In that setting, randomness enters through the selection of the map (or parameters) and the subsequent evolution is deterministic along each realisation.
From the viewpoint of random dynamical systems in the sense of Arnold~\cite{arnold1995random}, this corresponds to a restricted class of random systems rather than a cocycle driven by an ergodic base.
Accordingly, Andres adopts a notion of random periodic orbit that differs from ours; we recall it next for comparison.

\begin{definition}[\cite{andres2008randomization}]
\label{def:AndresRandomOrbit}
Let $A\subset X$ be a closed subset of a metric space $X$ and consider a
random operator 
\[
  \varphi:\Omega\times A \to \mathcal{P}(X),
\]
that is, $\varphi$ is measurable with respect to the product
$\sigma$–algebra and takes values in the closed subsets of $X$.

A sequence of measurable maps $\{\nu_{i}\}_{i=0}^{k-1}$, with $\nu_{i}:\Omega\rightarrow A$ for all $i\in\{0,\dots,k-1\}$, is called a \emph{$k$–periodic random orbit} associated to $\varphi$ if:

\begin{enumerate}
\item $\nu_{i+1}(\omega)\in\varphi(\omega,\nu_{i}(\omega))$ for all $i\in\{0,\dots,k-2\}$, and 
      $\nu_{0}(\omega)\in\varphi(\omega,\nu_{k-1}(\omega))$ for $\mathbb{P}$–almost all $\omega\in\Omega$;

\item the $k$–random orbit is \emph{not} obtained by repeating a shorter random orbit, i.e., 
      there do not exist integers $m,p\in\mathbb{N}$ with $p>1$ and $mp=k$ and measurable maps 
      $\mu_{0},\dots,\mu_{m-1}:\Omega\to A$ such that 
      \[
        \nu_{i}(\omega)=\mu_{i \bmod m}(\omega)
        \quad\text{for all } i=0,\dots,k-1
        \text{ and for $\mathbb{P}$–a.e. }\omega\in\Omega;
      \]

\item for all distinct $i,j\in\{0,\dots,k-1\}$ we have
$\nu_{i}(\omega)\neq\nu_{j}(\omega)$ for $\mathbb{P}$–almost all
$\omega\in\Omega$.

\end{enumerate}
\end{definition}

Intuitively, this definition lifts the classical concept of a periodic orbit from the phase space to the space of measurable functions. Instead of a finite set of points permuted by a single map, one considers a sequence of random variables $\nu_0, \dots, \nu_{k-1}$ that are mapped cyclically into one another by the random operator $\varphi$ for almost every realisation $\omega$. For $\mathbb{P}$-almost all $\omega$, the realisation $\varphi(\omega,\cdot)$ then admits a deterministic periodic 
orbit given by the values $\{\nu_0(\omega), \dots, \nu_{k-1}(\omega)\}$.

In \cite{andres2008randomization} a random modification of the Sharkovskiĭ's theorem is proved, stated in the next theorem: 
\begin{thm}[\cite{andres2008randomization}]
\label{thm:AndresShark} 
Suppose that $f(\cdot,x):\Omega\rightarrow\mathbb{R}$ is measurable for
all $x\in\mathbb{R}$ in a complete measurable space $\Omega$, and that
$f(\omega,\cdot):\mathbb{R}\rightarrow\mathbb{R}$ is continuous for
almost all $\omega\in\Omega$. If for some $k\in\mathbb{N}$ the map $f$
possesses a $k$–random orbit, then for every $n\in\mathbb{N}$ with
$k\prec n$ the map $f$ also admits an $n$–random orbit.

\end{thm}

The main difference between Theorem~\ref{thm:AndresShark} and our
result, Theorem~\ref{thm:main_forcing}, lies in the source of
randomness: in~\cite{andres2008randomization} the random choice occurs at the level of selecting the map, after which the system evolves
deterministically, whereas in our setting the noise acts continuously
through an ergodic driving base.

A closely related but more functional–analytic viewpoint is developed by
Barbarski~\cite{Barbarski2011}, who reformulates Andres’ random
Sharkovskiĭ theorem in the general setting of measurable (random)
operators and characterizes the existence of random orbits via
Kuratowski–Ryll-Nardzewski measurable selection theorems.

We next turn to the one–dimensional RDS framework developed by Klünger~\cite{klunger2001periodicity}. An adaptation of the notion of periodic orbit to random dynamical systems was proposed there, in a formulation that is intrinsically one–dimensional: the phase space is $\mathbb{R}$ and the definition explicitly exploits the total order on $\mathbb{R}$.

\begin{definition}[Random periodic cycle]\label{orbitaAleKlunger}
Following \cite{klunger2001periodicity}, consider a random dynamical system
$\phi$ generated by $\varphi$ over a metric dynamical system
$(\Omega,\mathcal{F},\mathbb{P},\theta)$ with state space $\mathbb{R}$.

Let $k \in \mathbb{N}$ and let $x_i \colon \Omega \to \mathbb{R}$,
$i \in \{1,\dots,k\}$, be random variables such that
\[
  x_1(\omega) < x_2(\omega) < \dots < x_k(\omega)
  \quad \text{for $\mathbb{P}$-a.e.\ $\omega$}.
\]
We say that the family $(x_i)_{i=1}^k$ \textbf{forms a random periodic cycle
of period $k$} if there exists a deterministic permutation $\pi \in S_k$ such
that
\[
  \phi(1,\omega,x_i(\omega)) = x_{\pi(i)}(\theta\omega)
  \quad \text{for $\mathbb{P}$-a.e.\ $\omega$ and all } i \in \{1,\dots,k\}.
\]

Moreover, we say that these random variables form a \textbf{random periodic
cycle of minimal period $k$} if, in addition,
\[
  \phi(\ell,\omega,x_i(\omega)) \neq x_i(\theta^\ell\omega)
  \quad \text{for all } i \in \{1,\dots,k\} \text{ and } 1 \le \ell < k,
\]
on a set of positive $\mathbb{P}$-measure.
\end{definition}

The relation between the objects in Definitions~\ref{def_randomPeriodicPoint},~\ref{def_randomPeriodicOrbit} and~\ref{orbitaAleKlunger} is subtle. It is true that random periodic points, random periodic cycles and random periodic orbits are in a one to one relation if the periods are $1$, as observed in Remark 4.10 of \cite{klunger2001periodicity}. Further discussions regarding the relation of those objects is done in Theorem 4.15 of the same reference. We stress that, generally, the existence of a random periodic point of minimal period $k$ \textit{does not} imply the existence of a random periodic cycle of minimal period $k$. However, if $x$ is a random periodic point of minimal period $k$, there exists a random periodic orbit of minimal period $k$ that contains $x$ as long as $\theta^k$ is ergodic.

Within this framework, Klünger~\cite{klunger2001periodicity} establishes a random analogue of Sharkovskiĭ's theorem for real-valued random dynamical systems on $\mathbb{R}$, formulated in terms of random periodic cycles in the sense of Definition~\ref{orbitaAleKlunger}. 

\begin{thm}[\cite{klunger2001periodicity}] \label{thm:SharkKlunger} Consider a random dynamical system $\phi$ generated by $\varphi$ with
noise $\theta$ such that $\theta^{n}$ is ergodic for all
$n\in\mathbb{N}$, and let $\{x_{1}\dw,\dots,x_{k}\dw\}$ be a random
periodic cycle of minimal period $k$, where
$x_{1}<\dots<x_{k}$ $\mathbb{P}$-a.s. Then, for every $n\in\mathbb{N}$
such that $k\prec n$, the system $\phi$ admits a random periodic orbit
of minimal period $n$ or $2n$.

\end{thm}

Theorem~\ref{thm:SharkKlunger} thus establishes a random analogue of the classical forcing relation: the existence of a cycle of minimal period $k$ implies the existence of random periodic orbits for all subsequent periods in the ordering, subject to the intrinsic ambiguity between $n$ and $2n$. In this sense, Klünger's result validates the Sharkovskiĭ hierarchy for a broad class of one-dimensional random dynamical systems on $\mathbb{R}$, modulo this specific period-doubling indeterminacy.

By contrast, our main result (Theorem~\ref{thm:main_forcing}) employs a
geometric, Conley-index-based argument tailored to small random
perturbations of interval maps. The trade-off is that we must impose
more restrictive structural assumptions on the underlying deterministic
map. However, this geometric strategy allows us to realize arbitrary
finite tails of the Sharkovskiĭ ordering with \emph{precise} control
over the periods, thereby removing the period-doubling uncertainty
present in Theorem~\ref{thm:SharkKlunger} and highlighting the
topological robustness of the orbits we obtain.

Beyond the random periodic orbits considered in
\cite{andres2008randomization, klunger2001periodicity}, we introduce a
more geometric refinement, the $(\delta,k)$–random periodic orbit
(Definition~\ref{def:epsilonkorbitaale}). In contrast with the notion in
\cite{klunger2001periodicity}, which yields a random periodic object
without controlling its spatial displacement, an $(\delta,k)$–random
periodic orbit requires the trajectory to return within an
$\delta$–neighbourhood of the deterministic periodic orbit at
prescribed times, while still allowing intermediate excursions (see
Figure~\ref{fig:epsilon_k_orbit}).


During the intermediate iterates, the point is carried away from this neighbourhood, so it returns there only after $n$ iterates and not before. Note that for a deterministic dynamical system, i.e.\ when the random noise is switched off, any periodic orbit of period $n$ is automatically an $(\delta,n)$–random periodic orbit for every $\delta>0$.

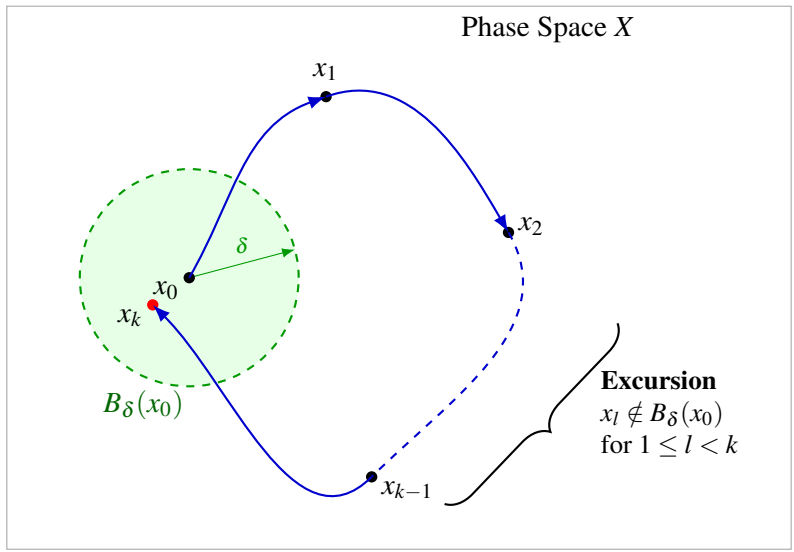
\begin{figure}[ht]
    \centering
\begin{tikzpicture}[>=Latex, scale=1.2]
    \tikzset{
        point/.style={circle, fill=black, inner sep=1.5pt},
        traj/.style={->, thick, blue!80!black, smooth, tension=0.7},
        epsball/.style={fill=green!10, draw=green!60!black, dashed, thick}
    }

    \draw[gray!70, fill=gray!0] (-2,-3) rectangle (6.6,3);
    \node[ anchor=north east] at (5,3) {Phase Space $X$};

    \coordinate (Center) at (0,0);
    \def\radius{1.2}
    \draw[epsball] (Center) circle (\radius);
    \node[green!40!black] at (-0.5, -1.4) {$B_\delta(x_0)$};
    
    \draw[->, green!60!black, thin] (Center) -- node[above, font=\scriptsize] {$\delta$} (15:\radius);

    \coordinate (x0) at (0., 0.);
    \node[point, label={[below left]:$x_0$}] at (x0) {};

    \coordinate (x1) at (1.5, 2.0);
    \node[point, label={[above]:$x_1$}] at (x1) {};

    \coordinate (x2) at (3.5, 0.5);
    \node[point, label={[right]:$x_2$}] at (x2) {};

    \coordinate (xk_minus_1) at (2.0, -2.2);
    \node[point, label={[below right]:$x_{k-1}$}] at (xk_minus_1) {};

    \coordinate (xk) at (-0.4, -0.3);
    \node[point, fill=red, label={[below left]:$x_k$}] at (xk) {};

    \draw[traj] (x0) to[out=60, in=200] (x1);
    \draw[traj] (x1) to[out=20, in=120] (x2);
    
    \draw[thick, blue!80!black, dashed] (x2) to[out=-60, in=45] (xk_minus_1);
    
    \draw[traj] (xk_minus_1) to[out=225, in=-45] (xk);

    \draw[decorate, decoration={brace, amplitude=10pt, mirror}, thick,] 
        (2.8, -2.5) -- (4.7, -0.5) node[midway, right=22pt, align=left, font=\small] 
        {\textbf{Excursion}\\ $x_l \notin B_\delta(x_0)$\\ for $1 \le l < k$};

\end{tikzpicture}
\caption{Geometric illustration of an $(\delta, k)$-random periodic orbit. The trajectory starts at $x_0$ inside the neighborhood $B_\delta$, performs an excursion outside the neighborhood for intermediate steps $l \in \{1, \dots, k-1\}$, and returns to $B_\delta$ at step $k$. This defines a form of periodicity with spatial tolerance $\delta$.}
    \label{fig:epsilon_k_orbit}
\end{figure}

The main difference  between Definition \ref{def:epsilonkorbitaale} and the Definition \ref{orbitaAleKlunger} posed in \cite{klunger2001periodicity} is the geometrical aspect of a periodic orbit, since the latter just guarantees the existence of a random variable but that would not immediately give any information about the physical displacement of the point by the random map. 

Beyond forcing implications, our work also addresses a realisation problem in the random setting.  The results of Klünger and Andres are primarily forcing theorems: they show that the existence of period $k$ implies the existence of further periods, but do not construct systems realizing a prescribed finite Sharkovskiĭ tail as a set of minimal periods. In contrast, Theorem~\ref{thm:main_forcing} gives a constructive existence result: by perturbing deterministic maps with known periodic structure, we obtain random systems whose minimal periods form an arbitrary finite Sharkovskiĭ tail.

Finally, our proof strategy differs substantially from the existing
ones. Klünger’s arguments are based on ergodic theory and the
combinatorics of the order on $\mathbb{R}$, whereas Andres employs
functional–analytic tools such as fixed-point indices for multivalued
maps and measurable selection theorems. Our approach is topological: we use Conley index theory for random dynamical systems. Hyperbolic
deterministic periodic orbits give rise to isolated invariant sets with non-trivial cohomological Conley index, and the homotopy invariance of the index ensures that this topological signature persists under small random perturbations. By computing these indices we can follow such orbits into the random setting and control their minimal periods, thereby avoiding the $n$ versus $2n$ ambiguity that arises in purely
measure-theoretic approaches.

\section{Conley index} \label{DefConleyDetAndRand}

The long-term behaviour of a dynamical system is often organized by a collection of invariant sets such as attractors and repellers. Conley index theory provides topological invariants for isolated invariant sets and allows one to distinguish and track such sets under perturbations.

Throughout this section, we assume that~$(X,d)$ is a Polish and locally compact metric space equipped with a dynamical system given by the iterations of~$f:X\rightarrow X$ a continuous function. 

\subsection{Deterministic Conley Index}
\label{conleyDeterministicCalculation}

 In this subsection, we summarise the definition of the Cohomological Conley index as originally introduced in \cite{mrozek1990leray}. The initial part of the construction follows closely the presentation in \cite{franks2000shift}. Our main reference, from which most of the definitions and proofs in this subsection are drawn, is \cite{villapouca2013teoria}. For a more detailed treatment of the topic, we refer the reader to that work. The next definition provides a concise reformulation of Definitions 2.1, 2.3, and 2.4 from \cite{franks2000shift}.

\begin{definition} \cite{franks2000shift}
We define the following:

\begin{itemize}
    \item For a subset~$N \subset X$, the \textbf{maximal invariant set} in~$N$ is defined by
    \[
        \mathrm{Inv}(N) := \{\, x \in N \mid f^{n}(x) \in N \text{ for all } n \in \mathbb{N} \,\}.
    \]

    \item A set~$S \subset X$ is called an \textbf{isolated invariant set} if there exists a compact neighbourhood~$N$ of~$S$ such that 
    \[
        S \subset \mathrm{int}(N) \quad \text{and} \quad S = \mathrm{Inv}(N).
    \]
    In this case,~$N$ is referred to as an \textbf{isolating neighbourhood} of~$S$.

    \item A compact set~$N$ is called an \textbf{isolating block} if 
    \[
        f(N) \cap N \cap f^{-1}(N) \subset \mathrm{int}(N).
    \]

    \item The \textbf{exit set} of~$N$ is defined by
    \[
        N^{-} := \{\, x \in N \mid f(x) \notin \mathrm{int}(N) \,\}.
    \]
\end{itemize}
\end{definition}

Note that every isolating block is an isolating neighbourhood for the set $S=\text{Inv}(N)$ since $\text{Inv}(N)\subset f(N)\cap N\cap f^{-1}(N)\subset \text{int}(N)$ (this is observed in Chapter 2.1 of \cite{villapouca2013teoria}). This offers a useful tool for identifying isolating neighbourhoods, as the definition of an isolating block can be is easier to verify. The next definition is Definition 3.1 of \cite{franks2000shift}.

\begin{definition} \cite{franks2000shift}
    Let $S$ be an invariant isolated set with $N$ isolating neighbourhood. Assume that $L,N$ are two compact subsets of $X$ with $L\subset N$ both contained in the interior of the domain of $f$. The pair $(N,L)$ is called a \textbf{filtration pair} for $S$ if both $N$ and $L$ are the closure of their interiors and the following properties are verified:

    \begin{itemize}
        \item $\overline{N\setminus L}$ is an isolating neighbourhood of $S$, i.e., $S=\text{Inv}(\overline{N\setminus L})\subset \text{int}(\overline{N\setminus L})$
        \item $L$ is a neighbourhood of $N^-$ in $N$
        \item $f(L)\cap \overline{N\setminus L}= \emptyset$
    \end{itemize}
\end{definition}

For the next theorem, an additional definition is required to establish the notion of distance in the space of functions.

\begin{definition} 
    \textbf{($C^0$-topology)} Let $X$ and $Y$ be two topological spaces and $C^0(X,Y)$ the space of functions of class $C^0$ from $X$ to $Y$. Consider the sets given by the form:
    \begin{equation}
        A(K,U) = \{f\in C^0(X,Y)\,:\,f(K)\subset U\}
    \end{equation}
in which $K\subset X$ is compact and $U\subset Y$ is open. The topology obtained by taking those sets is called $C^0$-topology. 
\end{definition}

The equivalence of the $C^0$ topology and the uniform norm on the space of continuous functions defined on a compact interval $[a, b]$ can be found in Proposition 43.7 of \cite{willard2012general}. This is a key observation for this work.

\begin{thm}
     Let $P=(N,L)$ be a filtration pair for $f$ and let $N_L$ the quotient space $N/L$ in which the collapsed set $L$ is denoted by $[L]$ and it is taken as the base point. Then, $f$ induces a continuous map $f_P:N_L\rightarrow N_L$ that preserves the base point and also has the property that $[L]\subset \text{int}(f^{-1}[L])$. The map $f_P$ is called the pointed space application associated to $P$.
\end{thm}

The next theorem, in a certain sense, ensures that the preceding definitions are indeed meaningful, confirming the existence of the objects defined earlier. The result is established in \cite{franks2000shift}, specifically in the proof of Theorems 3.6 and 3.8 from that work.

\begin{thm} \cite{franks2000shift}\label{existenceFiltPair}
    Let $S$ be an invariant isolated set. Then, it follows that:
    \begin{itemize}
        \item Every neighbourhood of $S$ contains an isolating block $N$. 
        \item If $L$ is any compact neighbourhood of $N^-$ in $N$ sufficiently small, then $(N,L)$ is a filtration pair. 
        \item There exists a neighbourhood of $f$ in the $C^0$-topology such that for any $\Tilde{f}$ in this neighbourhood, $\Tilde{S}=\text{Inv}(N)$ with respect to $\Tilde{f}$ is an invariant isolated set and $(N,L)$ is a filtration pair for $\Tilde{S}$. 
    \end{itemize}

\end{thm}

The following definition is put in a wider sense in Section 4 of \cite{franks2000shift}. 

\begin{definition} \cite{franks2000shift} Let $X$ and $X'$ be two metric spaces. We say that $(X,f)$ and $(X',g)$ are shift-equivalent, which we denote by $f\sim_s g$, if there exists $m\in \mathbb{Z}^+$, $r:X\rightarrow X'$ and $s:X'\rightarrow X$ such that the following diagrams commute:

 \begin{center}
\begin{minipage}{0.40\textwidth}
\centering
\begin{tikzpicture}[auto, node distance=2cm, thick]
    
    \node (A) at (0, 0)  {$X$};
    \node (B) at (2, 0)  {$X$};
    \node (C) at (2, -2)  {$X'$};
    \node (D) at (0, -2)  {$X'$};
    
    \draw[->] (A) -- node[above] {$f$} (B);
    \draw[->] (B) -- node[right] {$r$} (C);
    \draw[->] (D) -- node[above] {$g$}(C);
    \draw[->] (A) --  node[right] {$r$} (D);

\end{tikzpicture}
\end{minipage}
\hspace{0cm}
\begin{minipage}{0.40\textwidth}
\centering
\begin{tikzpicture}[auto, node distance=2cm, thick]
    
    \node (A) at (0, 0)  {$X'$};
    \node (B) at (2, 0)  {$X'$};
    \node (C) at (2, -2)  {$X$};
    \node (D) at (0, -2)  {$X$};
    
    \draw[->] (A) -- node[above] {$g$} (B);
    \draw[->] (B) -- node[right] {$s$} (C);
    \draw[->] (D) -- node[above] {$f$}(C);
    \draw[->] (A) --  node[right] {$s$} (D);
    
\end{tikzpicture}
\end{minipage}
\end{center}
and also such that the compositions satisfy $r\circ s = g^m$ and $s\circ r = f^m$. 
\end{definition}

Consider $h_P(S)$ the homotopy class of the applications that preserve the base point in $N_L$. From Theorem \ref{existenceFiltPair}, one can take $f_P$ as the representative of $h_P(S)$. The next theorem is a central result regarding the Conley index theory. It is Theorem 4.3 of \cite{franks2000shift}, where its proof can be found. 

\begin{thm} \cite{franks2000shift}
    Assume that $P_1 = (N_1,L_1)$ and $P_2 = (N_2,L_2)$ are two filtration pairs for $S$. Then, the induced applications $f_{P_1}$ and $f_{P_2}$, on the corresponding pointed spaces, are shift equivalent. 
\end{thm}

There are several approaches to defining the Conley index. In this work, we adopt the formulation introduced as a pair in \cite{mrozek1990leray} in Definition 2.8 of that work. However, alternative methods exist. For example, Definition 4.8 of \cite{franks2000shift} introduces the discrete homotopy Conley index. Furthermore, \cite{villapouca2013teoria} provides a comprehensive overview of various definitions in Chapter 2. This latter reference is regarded as an extensive and thorough source on the subject of Conley index definitions.

\begin{definition} \cite{villapouca2013teoria}
    Let $S$ be an isolated invariant set for a continuous map $f$. 
    
    The \textbf{homotopical Conley index} $h(S,f)$ is defined as the shift-equivalence class of $h_P(S)$, where $P=(N,L)$ is a filtration pair for $S$. Let $\text{Con}^*(S)$ denote the \textbf{cohomological Conley index} which is defined as the shift equivalence class of $(f_P)^*:H^*(N_L,[L])\rightarrow H^*(N_L,[L])$ where $H^*(N_L,[L])$ is the cohomology of $N_L$. The \textbf{reduced cohomological Conley index} $\text{Con}^*(S)$ will be defined as a pair $(CH^*(S),\chi^*(S))$ where $CH^*(S) = \cap_{n>0}(f_P^*)^n(H^*(N_L,[L]))$ and $\chi^*(S):CH^*(S)\rightarrow CH^*(S)$ is an automorphism induced by $f^*_P$. 
\end{definition}

\begin{obs} \label{homoClassOfPoint}
    (\cite{szymczak1995conley}) A pointed space $Y$ with an application $g:Y\rightarrow Y$ is such that the homotopy class $(Y,[g])=0$ if, and only if, the map $g^n$ is homotopically equivalent to a constant for some $n\in \mathbb{Z}^+$
\end{obs}

In \cite{mrozek1990leray}, several properties of the Conley index are established. Among these, Proposition 2.10 and Theorem 2.11 from this reference will be utilised in the present work. To ensure the text remains self-contained, these results are stated below:

\begin{thm} \textbf{(Wazewski Property)}
    A nonzero reduced cohomological Conley index implies a non-empty isolated invariant set.
\end{thm}

\begin{thm}[\textbf{Continuation Property}]
Let $f\colon X\to X$ be a continuous map, let $N\subset X$ be an isolating neighbourhood for $f$, and set $S:=\Inv(N,f).$
Then there exists a neighbourhood $\mathcal U$ of $f$ in the $C^{0}$-topology such that, for every
$\tilde f\in\mathcal U$, the set $N$ is an isolating neighbourhood for $\tilde f$ and, writing
$\tilde S:=\Inv(N,\tilde f),$
one has invariance of the cohomological Conley index:
\[
Con^{*}(S,f)=Con^{*}(\tilde S,\tilde f).
\]
\end{thm}

\subsection{Conley Index calculations}

In this subsection we compute the Conley index for a hyperbolic $k$--cycle.
The key point is that this index is nontrivial, and we will later use this
nontriviality as the mechanism that forces the existence of invariant dynamics
inside a prescribed isolating neighbourhood.

\begin{thm}\label{GenMapsConleyIndex}
Let $C\subset\mathbb{R}$ be a compact interval and let $f\colon C\to C$ be
$C^{1}$ on a neighbourhood of a periodic orbit
$S=\{p_{1},\dots,p_{k}\}$ of (prime) period $k\in\mathbb{N}$, with
$f(p_i)=p_{i+1}$ for $i\in\{1,\dots,k-1\}$ and $f(p_k)=p_1$.
Set $g:=f^{k}$. Assume that $S$ is \emph{hyperbolic}, i.e.
\[
|g'(p_i)|\neq 1\quad\text{for all }i\in\{1,\dots,k\}.
\]
Then there exist pairwise disjoint compact intervals
\[
N_i=[p_i-\epsilon_i,p_i+\epsilon_i]\qquad (i=1,\dots,k),
\]
with $\epsilon_i>0$, such that, with $N=\bigcup_{i=1}^{k}N_i$, the set $N$ is an
isolating neighbourhood for $S$.

Moreover, exactly one of the following two cases holds:

\smallskip
\noindent\textbf{ (A$_k$) Attracting cycle.} If $|g'(p_i)|<1$ for all $i$, then
\begin{equation}\label{conleyGenMapping_attractor}
\mathrm{Con}^{q}(S)=
\begin{cases}
\left(\displaystyle\bigoplus_{i=1}^{k}\mathbb{Q},\
\begin{bmatrix}
0 & 1 & 0 & \dots & 0 \\
0 & 0 & 1 & \dots & 0 \\
\vdots & \vdots & \vdots & \ddots & \vdots \\
0 & 0 & 0 & \dots & 1 \\
1 & 0 & 0 & \dots & 0
\end{bmatrix}\right), & \text{if }q=0,\\[2mm]
0, & \text{otherwise.}
\end{cases}
\end{equation}

\smallskip
\noindent\textbf{ (R$_k$) Repelling cycle.} If $|g'(p_i)|>1$ for all $i$, then
\begin{equation}\label{conleyGenMapping_repulsor}
\mathrm{Con}^{q}(S)=
\begin{cases}
\left(\displaystyle\bigoplus_{i=1}^{k}\mathbb{Q},\
\begin{bmatrix}
0 & \sigma_1 & 0 & \dots & 0 \\
0 & 0 & \sigma_2 & \dots & 0 \\
\vdots & \vdots & \vdots & \ddots & \vdots \\
0 & 0 & 0 & \dots & \sigma_{k-1} \\
\sigma_k & 0 & 0 & \dots & 0
\end{bmatrix}\right), & \text{if }q=1,\\[2mm]
0, & \text{otherwise,}
\end{cases}
\end{equation}
where $\sigma_i=\mathrm{sign}(f'(p_i))\in\{+1,-1\}$.

\smallskip
In particular, in both cases one has $h(S)\neq 0$.
\end{thm}

\begin{proof}
Set $g:=f^{k}$. Since $|g'(p_i)|\neq 1$ for each $i$ and $g'$ is continuous,
for each $i$ there exists $\gamma_i'>0$ such that either $|g'(x)|<1$ for all
$x\in(p_i-\gamma_i',p_i+\gamma_i')$ (attracting case), or $|g'(x)|>1$ for all
$x\in(p_i-\gamma_i',p_i+\gamma_i')$ (repelling case).

Choose $0<\epsilon_i<\gamma_i'$ so that the compact intervals
\[
N_i:=[p_i-\epsilon_i,p_i+\epsilon_i]
\]
are pairwise disjoint, and set $N:=\bigcup_{i=1}^k N_i$.
Shrinking the $\epsilon_i$ if necessary, we may also assume that
\begin{equation}\label{eq:itinerary-forcing}
f(N_i)\cap N \subset \mathrm{int}(N_{i+1})\qquad (i=1,\dots,k),
\end{equation}
with indices modulo $k$ (so $N_{k+1}=N_1$).

\medskip
We first claim that~$N$ is isolating and $\Inv_f(N)=S$ in both cases. We prove $\Inv_f(N)=S$ by working with the return map $g=f^k$.

\smallskip
\noindent\emph{(A$_k$) attracting case.}
Shrinking the $\epsilon_i$ if necessary, since $|g'(p_i)|<1$ and $g'$ is continuous,
there exists $\beta\in(0,1)$ such that
\[
|g'(x)|\le 1-\beta\qquad\text{for all }x\in N.
\]
Fix $i$ and $x\in N_i$. By the mean value theorem,
\[
|g(x)-g(p_i)|\le (1-\beta)|x-p_i|<|x-p_i|.
\]
Since $g(p_i)=p_i$, we obtain $g(N_i)\subset \mathrm{int}(N_i)$ for every $i$.

Now let $x\in\Inv_f(N)$, and choose an  orbit $(x_n)_{n\in\mathbb{Z}}\subset N$
with $f(x_n)=x_{n+1}$ and $x_0=x$. Define $y_m:=x_{mk}$ for $m\in\mathbb{N}$.
Then $y_{m+1}=g(y_m)$ and $(y_m)_{m\in\mathbb{Z}}\subset N$ is a full $g$-orbit.
If $y_m\in N_i$ set $\delta_m:=|y_m-p_i|$. Then
\[
\delta_{m+1}=|g(y_m)-p_i|\le (1-\beta)\delta_m.
\]

Iterating backwards gives $\delta_0\le (1-\beta)^{n}\delta_{-n}$ for all $n\ge 1$,
hence letting $n\to\infty$ yields $\delta_0=0$, so $y_0=p_i$. Therefore
$x=x_0$ lies on the $k$-cycle $S$. This proves $\Inv_f(N)=S$, so $N$ is an
isolating neighbourhood for $S$.

\smallskip
\noindent\emph{(R$_k$) repelling case.}
For each $i$, after shrinking $\gamma_i'$ if necessary we have
$|g'(x)|>1$ on $(p_i-\gamma_i',p_i+\gamma_i')$, hence $g'(x)\neq 0$ there.
Therefore $g'$ has constant sign on that interval and $g$ is monotone on it.

Shrinking the $\epsilon_i$ if necessary, since $|g'(p_i)|>1$ and $g'$ is continuous,
there exists $\beta>0$ such that
\[
|g'(x)|\ge 1+\beta\qquad\text{for all }x\in N.
\]

Assume $x\in N\setminus S$ and that there exists a full orbit
$(x_n)_{n\in\mathbb{Z}}\subset N$ with $x_0=x$.
As above set $y_m:=x_{mk}$, so $y_{m+1}=g(y_m)$ and $(y_m)\subset N$.
Pick $m$ and $i$ with $y_m\in N_i$, write $y_m=p_i+\alpha$ with $\alpha\neq 0$,
and define $\delta_m:=|y_m-p_i|=|\alpha|$. By the mean value theorem and the
lower bound on $|g'|$,
\[
\delta_{m+1}=|g(y_m)-p_i|=|g(y_m)-g(p_i)|\ge (1+\beta)|y_m-p_i|=(1+\beta)\delta_m.
\]

Iterating forward yields $\delta_{m+n}\ge (1+\beta)^n\delta_m\to\infty$ as $n\to\infty$,
contradicting $y_{m+n}\in N$ for all $n$ (since $N$ is bounded). Hence no such
full orbit exists for any $x\in N\setminus S$, so $\Inv_f(N)=S$, and $N$ is
isolating for $S$.

\medskip
\noindent\textit{Filtration pair and index in the attracting case.}
Assume (A$_k$). Since $f$ is continuous and $f(p_i)=p_{i+1}$, we may (after
shrinking the $\epsilon_i$ if necessary) arrange
\[
f(N_i)\subset \mathrm{int}(N_{i+1})\qquad (i=1,\dots,k),
\]
where $N_{k+1}=N_1$. Then $f(N)\subset \mathrm{int}(N)$, so the
exit set of $N$ is empty and we may take $L=\emptyset$; thus $(N,\emptyset)$ is a
filtration pair for $S$.

The pointed space $N/\emptyset$ is $N$ with a disjoint basepoint; since $N$ has
$k$ connected components, the reduced cohomology satisfies
\[
\widetilde H^{0}(N/\emptyset;\mathbb{Q})\cong \bigoplus_{i=1}^{k}\mathbb{Q},
\qquad
\widetilde H^{q}(N/\emptyset;\mathbb{Q})=0\ \text{ for }q\ge 1.
\]

Because $f$ maps the $i$-th component into the $(i+1)$-th, the pullback
$f^{*}$ on $\widetilde H^{0}$ cyclically permutes the basis in the
\emph{reverse} direction (mapping the generator of $N_{i+1}$ to the generator of $N_i$),
yielding the permutation matrix in \eqref{conleyGenMapping_attractor}.

\medskip
\noindent\textit{Filtration pair and index in the repelling case.}
Assume (R$_k$). In this case $|g'(p_i)|>1$ implies $g'(p_i)\neq 0$, so
\[
g'(p_i)=\prod_{j=0}^{k-1} f'(p_{i+j})\neq 0,
\]
and hence $f'(p_j)\neq 0$ for every point $p_j$ on the periodic orbit. By
continuity of $f'$, after shrinking the $\epsilon_i$ if necessary we may assume
that $f'$ has constant sign and no zeros on each $N_i$, so that $f$ is strictly
monotone on $N_i$ and $\sigma_i=\mathrm{sign}(f'(p_i))\in\{\pm1\}$ is well-defined.

We first build a filtration pair for the return map $g$ on each
component $N_i$ (where $g(p_i)=p_i$).

Fix $i$. Since $g$ is monotone on $N_i$ and $|g'(x)|\ge 1+\beta$ on $N_i$,
there exist points $q_{i,g}^{-}\in(p_i-\epsilon_i,p_i)$ and
$q_{i,g}^{+}\in(p_i,p_i+\epsilon_i)$ such that
\[
\{g(q_{i,g}^{-}),\, g(q_{i,g}^{+})\}=\{p_i-\epsilon_i,\, p_i+\epsilon_i\},
\]
and the set
\[
N_{i,g}^{-}:=[p_i-\epsilon_i,q_{i,g}^{-}]\cup[q_{i,g}^{+},p_i+\epsilon_i]
\]
satisfies $g(N_{i,g}^{-})\cap \mathrm{int}(N_i)=\emptyset$ while
$g(N_i\setminus N_{i,g}^{-})\subset \mathrm{int}(N_i)$. Let
\[
N_g^{-}:=\bigcup_{i=1}^{k}N_{i,g}^{-}.
\]

We can enlarge $N_g^{-}$ to a compact neighbourhood
$L_g\subset N$ such that $(N,L_g)$ is a filtration pair for $S$ with respect to
$g$, i.e.
\begin{equation}\label{eq:g-filtration}
g(L_g)\cap \overline{N\setminus L_g}=\emptyset,
\qquad
\Inv_g(\overline{N\setminus L_g})=S.
\end{equation}

We now define a filtration pair for $f$. Set
\begin{equation}\label{eq:L-from-Lg-fixed}
L:=\Bigl(\; \bigcup_{r=0}^{k-1}\bigl(f^{r}(L_g)\cap N\bigr)\ \cup\
\bigcup_{j=1}^{k-1}\bigl(f^{-j}(L_g)\cap N\bigr)\;\Bigr).
\end{equation}

Then $L$ is compact. In particular, $L$ contains a neighbourhood of the $g$-exit set
and of its forward and backward $f$-iterates up to time $k-1$.

We check the filtration pair property.

\smallskip
\noindent\emph{(i) $f(L)\cap \overline{N\setminus L}=\emptyset$.}
Let $x\in L$.
If $x\in f^{-j}(L_g)\cap N$ for some $1\le j\le k-1$, then $f(x)\in f^{-(j-1)}(L_g)\cap N\subset L$.
If $x\in f^{r}(L_g)\cap N$ for some $0\le r\le k-2$, then $f(x)\in f^{r+1}(L_g)\cap N\subset L$.
Finally, if $x\in f^{k-1}(L_g)\cap N$, write $x=f^{k-1}(z)$ with $z\in L_g$; then
$f(x)=f^{k}(z)=g(z)$. If $f(x)\in N$, then by \eqref{eq:g-filtration} we have
$g(z)\notin \overline{N\setminus L_g}$, hence $g(z)\in L_g\subset L$; otherwise $f(x)\notin N$.
In all cases, $f(x)\notin \overline{N\setminus L}$, proving (i).

\smallskip
\noindent\emph{(ii) $\Inv_f(\overline{N\setminus L})=S$.}
Let $x\in \Inv_f(\overline{N\setminus L})$ and take a full orbit
$(x_n)_{n\in\mathbb{Z}}\subset \overline{N\setminus L}$ with $x_0=x$.
In particular $x_{mk}\notin L_g$ for all $m\in\mathbb{Z}$ (since $L_g\subset L$).
Define $y_m:=x_{mk}$. Then $y_{m+1}=g(y_m)$ and $(y_m)_{m\in\mathbb{Z}}\subset
\overline{N\setminus L_g}$ is a full $g$-orbit. Hence $y_0\in \Inv_g(\overline{N\setminus L_g})=S$
by \eqref{eq:g-filtration}, so $x=x_0$ lies on the cycle $S$. Thus
$\Inv_f(\overline{N\setminus L})=S$.

Therefore $(N,L)$ is a filtration pair for $S$ with respect to $f$.

Moreover, by the construction of $L$, on each component $N_i$ the set $L\cap N_i$
contains intervals covering a neighbourhood of both endpoints of $N_i$, so that
$N_i\setminus L$ is a single open interval. Collapsing $L$ to a point, each $N_i$
therefore produces a circle, and we obtain
\[
  N/L \simeq \bigvee_{i=1}^{k}S^{1}.
\]

Let $e_i$ denote the standard basis of $H^{1}(N/L;\mathbb{Q})\cong\bigoplus_{i=1}^k\mathbb{Q}$
dual to the $i$-th circle. Since $f$ is monotone near $p_i$ and maps the
$i$-th component to the $(i+1)$-th with degree $\sigma_i=\mathrm{sign}(f'(p_i))$,
the pullback on cohomology satisfies
\[
f^{*}(e_{i+1})=\sigma_i\, e_i,\qquad \sigma_i=\mathrm{sign}(f'(p_i))\in\{\pm 1\},
\]
(with indices taken modulo $k$). Hence, in the basis $(e_1,\dots,e_k)$, the index map is
represented by the matrix in \eqref{conleyGenMapping_repulsor}.
\end{proof}

\subsection{Random Conley Index}

Throughout all this section, the metric space $(X,d_X)$ is assumed to be Polish \footnote{Separable and complete.}. For $x\in X$, $A,B\subset X$ fixed, define $\text{dist}_X(x,B) = \inf_{y\in B}\{d_X(x,y)\}$ and $\text{dist}_X(A,B) = \inf_{x\in A}\{\text{dist}_X(x,B)\}$. 

\begin{definition} \cite{liu2008conley}
    A map from $\Omega$ to $\mathcal{P}(X) = \{A\,|\,A\subset X\}$ is called a \textbf{multifunction} from $\Omega$ to $X$. A multifunction $D$ is a \textbf{random set} if the map $\omega \mapsto \text{dist}_X(x,D(\omega))$ is measurable for all $x\in X$. In this case, we will also say that $D$ is measurable.   
    
    A measurable multifunction $D$ such that $\omega \mapsto D(\omega)$ takes values only in closed (respectively, compact) subsets of $X$ is called a \textbf{random closed set} (respectively, \textbf{random compact set}). A multifunction $U$ such that $\omega\mapsto U(\omega)$ takes value only in open sets is called a \textbf{random open set} if $\omega \mapsto U^C(\omega)$ is a random closed set. 
\end{definition}

The following definitions introduce the necessary components for constructing the random Conley index. These are provided here to make this work as self-contained as possible, though for a comprehensive treatment of this topic, refer to \cite{liu2008conley}. All of the forthcoming definitions can be found in Section 3 of that reference.

\begin{definition}
Let $\phi\colon \mathbb{N}\times\Omega\times X\to X$ be a random dynamical system
and let $N=(N(\omega))_{\omega\in\Omega}$ be a random compact subset of $X$.
For a given $\omega\in\Omega$, a sequence $(x_n)_{n\in\mathbb{Z}}$ taking values in~$X$ is called a
\textbf{full orbit in $N$ over $\omega$} if
\[
  x_n\in N(\theta^n\omega)\quad\text{and}\quad
  \phi\bigl(1,\theta^n\omega,x_n\bigr)=x_{n+1}
  \quad\text{for all }n\in\mathbb{Z}.
\]
The \textbf{maximal random invariant set} in $N$ is then defined by
\[
  \Inv(N,\phi)(\omega)
  :=\bigl\{x\in N(\omega)\,:\,\exists\text{ a full orbit }
  (x_n)_{n\in\mathbb{Z}}\text{ in }N\text{ over }\omega
  \text{ with }x_0=x\bigr\}.
\]
Equivalently, $\Inv(N,\phi)$ is the largest random invariant set contained in
$N$ with respect to inclusion.
\end{definition}

If the context makes it clear which dynamical system is being referred to, $\text{Inv}(N, \phi)(\omega)$ may be denoted simply by $\text{Inv}(N)(\omega)$. 

\begin{definition}[Random invariant set and random isolating neighbourhood]
    A random compact set $N$ is said to be a \textbf{random isolating neighbourhood} if it satisfies $\text{Inv}(N,\phi)\subset \text{int}(N)$. A random set $S$ is called a \textbf{random isolated invariant set} if there exists a random isolating neighbourhood $N$ such that $S(\omega) = \text{Inv}(N,\phi)(\omega)$ $\mathbb{P}$-a.s. for all $\omega \in \Omega$. 
\end{definition}

\begin{definition}[Random isolating block and exit set ]
Let $\varphi$ be a (discrete-time) random dynamical system on a compact metric space $X$ over
$(\Omega,\mathcal F,\mathbb P,\theta)$, and write the time--one map as
\[
\varphi(\omega,\cdot):=\varphi(1,\omega,\cdot)\colon X\to X.
\]
A random compact set $N(\omega)$ is called a \textbf{random isolating block} if, for
$\mathbb P$-a.e.\ $\omega\in\Omega$,
\begin{equation}\label{eq:random-isolating-block}
\varphi(\theta_{-1}\omega,\,N(\theta_{-1}\omega))
\;\cap\; N(\omega)
\;\cap\; \varphi(\omega,\cdot)^{-1}\!\bigl(N(\theta\omega)\bigr)
\;\subset\; \operatorname{int} N(\omega),
\end{equation}
where
\[
\varphi(\omega,A):=\{\varphi(\omega,x):x\in A\},\qquad
\varphi(\omega,\cdot)^{-1}(A):=\{x\in X:\varphi(\omega,x)\in A\}.
\]

Given a random isolating neighbourhood (in particular, a random isolating block) $N(\omega)$,
its \textbf{random exit set} is
\begin{equation}\label{eq:random-exit-set}
N^{-}(\omega):=\Bigl\{x\in N(\omega)\;:\;\varphi(\omega,x)\notin \operatorname{int}N(\theta\omega)\Bigr\},
\qquad \mathbb P\text{-a.s.}
\end{equation}
\end{definition}

Let $\phi$ be a discrete random dynamical system generated by a random map $\varphi$. It follows that every random isolating block is also a random isolating neighbourhood, as noted in Remark 3.2 of \cite{liu2008conley}. The next definitions can also be found in Section 4 of \cite{liu2008conley}.

\begin{definition}
    \textbf{(Random Filtration Pair)} Let $N$ be a random isolating neighbourhood, $L(\omega)\subset N(\omega)\mathbb{P}\text{-a.s.}$ a compact random set  and $S(\omega)$ a random invariant isolated set inside $N(\omega)$. Assume that $N(\omega) = \text{cl(int($N(\omega$))}$ and $L(\omega) = \text{cl(int($L(\omega$))}$. We say that $(N(\omega), L(\omega))$ is a \textbf{random filtration pair} for $S$ if the following conditions are satisfied $\mathbb{P}$-a.s.:

    \begin{itemize}
        \item cl(int($N(\omega)\setminus L(\omega)$)) is a random isolating neighbourhood of $S(\omega)$,
        \item $L(\omega)$ is a random neighbourhood of $N^-(\omega)$ in $N(\omega)$,
        \item $\varphi(\omega, L(\omega))\cap\text{cl}(N\setminus L)(\theta \omega)=\emptyset$.
    \end{itemize}
\end{definition}

The following definition introduces a metric to measure the distance between random maps. This metric is essential for ensuring that the random Conley index remains invariant under small perturbations, as it provides a precise framework for determining what constitutes a ``small'' change.

\begin{definition}
    Consider two random maps $\varphi(\omega)$ and $\psi(\omega)$ and $N(\omega)$ a random compact set. Define:

    \begin{equation}
        d^N_\omega(\varphi,\psi) = \sup_{z\in N(\theta_{-1}\omega)}d_X(\varphi(\theta_{-1}\omega)(z), \psi(\theta_{-1}\omega)(z))
    \end{equation}

    The function $d^N_{\cdot}(\cdot, \cdot)$ is called a \textbf{random metric} in $\mathcal{M}$ and $(\mathcal{M}, d^N)$ is called a \textbf{random metric space}. A set $\mathcal{N}$ is called a $C^0$-random neighbourhood of $\varphi \in \mathcal{M}$ if there exists a random variable $r>0$ such that:

    \begin{equation}
        B_r(\varphi)=\{f\in \mathcal{M}\,|\,d_{\omega}^N(f,\varphi)<r(\omega),\mathbb{P}\text{-a.s.}\}\subset\mathcal{N}
    \end{equation}

    The pair $(\mathcal{M}, \tau)$ is called $C^0$\textbf{-random topology} if the topology $\tau$ in $\mathcal{M}$ is the one induced of the $C^0$-random neighbourhood. 
\end{definition}

The following result, Theorem 4.1 from \cite{liu2008conley}, serves a similar purpose to Theorem \ref{existenceFiltPair}, but in the context of random dynamical systems. It thus validates the previous definitions by ensuring the existence of the corresponding objects.

\begin{thm} \cite{liu2008conley}\label{ingredient1}
    Suppose that $N(\omega)$ is a random isolating block and $L(\omega)$ is a random compact isolating neighbourhood of $N^-(\omega)$ in $N(\omega)$ sufficiently small. Then, $(N,L)(\omega)$ is a random filtration pair for $\text{Inv}(N,\varphi)(\omega)$. 

    Besides that, there exists a $C^0$-random neighbourhood of $\varphi$ in $\mathcal{M}$ such that, for any random maps $\psi$ in this neighbourhood, it follows that $S_\psi = \text{Inv}(N\setminus L, \psi)$ is a random isolated invariant set for $\psi$ and $(N,L)$ is a random filtration pair for $S_{\psi}$. 
\end{thm}

The definition of random shift equivalence is made in \cite{liu2008conley}; more specifically in Definition 5.1 of that work. We will omit it from here in order for our work to be concise; we will also omit the definition of random homotopical equivalence and the definition of the random shift equivalence class of the random pointed space applications; they can be found again in \cite{liu2008conley} precisely in Definitions 6.1 and in a discussion in Section 6 of the same reference. 

\begin{definition}
    Suppose that $P(\omega) = (N,L)(\omega)$ is a random filtration-pair for $\varphi(\omega)$. Let $N_L(\omega)$ be the quotient space $N\setminus L (\omega)$, where $N_L(\omega) = ((N\setminus L)(\omega)\cup[L(\omega)], [L(\omega)])$ for all $\omega \in \Omega$. If $L(\omega)=\emptyset$, then $N_L(\omega) = N(\omega)\cup [\emptyset]$. We identify $N_L\setminus [L]$ with $N\setminus L$, where the class $L[\omega]$ is defined by the equivalence relation $x\sim y$ if and only if $x,y\in L(\omega)$. 

    The random set $N_L(\omega)$ is called a \textbf{random quotient space} if $N_L(\omega) = ((N\setminus L)(\omega)\cup [L(\omega)],[L(\omega)])$ for all $\omega\in \Omega$ regardless of the choice of the filtration pair $P(\omega) = (N,L)(\omega)$. 

    A map $\varphi_P(\omega):N_L(\omega)\rightarrow N_L(\theta \omega)$ is called a \textbf{random map of the pointed space} associated to $P$ if $\varphi_P(\omega):N_L(\omega)\rightarrow N_L(\theta \omega)$ satisfies:

    \begin{equation}
        \varphi_P(\omega)(x) = \begin{cases}
            [L(\theta \omega)], \text{ if }x=[L(\omega)] \text{ or }\varphi(\omega)(x)\notin N(\theta \omega) \\
        p(\theta \omega)(\varphi(\omega)(x)), \text{ otherwise}
        \end{cases}
    \end{equation}
where $p(\omega):N(\omega)\rightarrow N_L(\omega)$ is a random quotient map where, for each $\omega \in \Omega$ and $p(\omega):N(\omega)\rightarrow N_L(\omega)$:

\begin{equation}
    p(\omega)(x) = \begin{cases}
        x, \text{ if }x\in (N\setminus L)(\omega) \\
        [L(\omega)], \text{ if }x\in L(\omega)
    \end{cases}
\end{equation}
\end{definition}

\begin{thm}
    Suppose that $P\dw=(N,L)\dw$ and $P'\dw = (N',L')\dw$ are two random filtration pairs for $S\dw$. Then, the induced maps $\varphi_P:N_L\dw \rightarrow N_L\dw$ and $\varphi_{P'}:N'_{L'}\rightarrow N'_{L'}$ are random shift equivalents. 
\end{thm}

\begin{definition}
    Let $C$ be a random pointed space and $f$ a random map that preserves the base point such that $f\dw:C\dw\rightarrow C(\theta_n\omega)$ is continuous and $f(\cdot)(x)$ is measurable with $n=n\dw$. Let $\mathcal{M}_C,n$ be the space of all random maps $g\dw:C\dw\rightarrow C(\theta_n\omega)$ such that $g\dw(\cdot)$ is continuous and $g(\cdot)(x)$ is measurable with $n=n\dw$. 

    Suppose that $C=N_L$ with $L\subset N$. given $\omega\in\Omega$ arbitrary but fixed, define:

    \begin{equation}
        \Tilde{d}_{\omega}(x,y) = \begin{cases}
            d_X(x,y),\,\text{ if }x,y\in (N\setminus L)\dw \\
            0,\,\text{ if }x=y=[L\dw] \\
            \text{dist}_X(y,L\dw),\,\text{ if }x=[L\dw]\text{ and }y\in (N\setminus L)\dw
        \end{cases}
    \end{equation}
 
    For $f,g\in \mathcal{M}_{C,n}$, define:
    \begin{equation}\label{defRandomMetric}
       d_{\omega}(f,g):= \sup_{x\in C(\omega)} \tilde d_{\theta_n\omega}\bigl(f(\omega,x),g(\omega,x)\bigr).
    \end{equation}

    We can show that $\Tilde{d}_{\omega}(\cdot,\cdot)$ is a metric for the pointed space. With that in mind, one can conclude that $(\mathcal{M}_{C,n},d)$ is a random metric space with the random topology induced by the random metric given as in Eq. \eqref{defRandomMetric}. 
    
\end{definition}

\begin{definition}
    \textbf{(Random Conley Index)} Suppose that $\varphi$ is the time $1$ map for a discrete random dynamical system, $S$ is a random invariant isolated set for $\varphi$ and $P=(N,L)$ is a random filtration pair for $S$. Let $h_P(S,\varphi)$ be the random homotopy class $[\varphi_P]$ of the pointed space $N_L$ with $\varphi_P$ as a representative. The random shift equivalence class of $h_P(S,\varphi)$, denoted by $h(S,\varphi)$, is called the \textbf{random Conley index} for $S$. 
\end{definition}

We emphasize that, from this point on, we work with random dynamical systems
indexed by discrete time $\mathbb{N}$.  Liu~\cite{liu2008conley} develops the Conley index
for random homeomorphisms over $\mathbb{Z}$ and, in that general (possibly
noncompact) setting, invertibility is needed to control backward orbits and to
identify random isolated invariant sets via two–sided trajectories.  In the
present paper the phase space $X$ is a compact metric space and we only use
forward time $n\in\mathbb{N}$.  In this situation, the key ingredients of
Liu’s construction (existence of random isolating neighbourhoods and filtration
pairs, and the random continuity property) carry over verbatim from random
homeomorphisms to random continuous maps, since compactness of preimages of
random compact sets follows directly from continuity on $X$.  Therefore, in
the statement below we formulate the random continuity property for random
continuous maps $\varphi\colon \mathbb{N}\times\Omega\times X\to X$ on a
compact phase space.

\begin{thm} \label{ContinuationProperty}
    \textbf{(Random Continuation Property)} Suppose that $\varphi_t$ with $t\in [0,1]$ is a family of random map that depends continuously of $t$ (in a neighbourhood of the $C^0$-topology). If $N$ is a random isolating neighbourhood for each $\varphi_t$, then the random Conley index $h(S_t, \varphi_t)$ for $t$ is independent of $t\in [0,1]$, i.e., $h(S_t,\varphi_t) = h(S_0,\varphi_0)$ where $S_t=\text{Inv}(N,\varphi_t)$  with $t\in [0,1]$ is the random isolated invariant set for $\varphi_t$ in $N$. 
\end{thm}

\begin{thm} \label{wazeProperty}
    \textbf{(Random Wazewski property)} Suppose that $S\dw$ is a random invariant isolated set and that its random Conley index for $S\dw$ is non trivial, i.e., $h(S,\varphi)\neq \underset{\text{-}}{0}$. Then $S\neq \emptyset$ $\mathbb{P}$-a.s. given that $\theta$ is ergodic over $\mathbb{P}$. 
\end{thm}

\section{Proof of the main results}\label{ProofMainThm}

\subsection{Proof of Theorem \ref{thm:main_forcing} }
In this subsection, Theorem \ref{thm:main_forcing} is proved. The proof strategy is motivated by the approach outlined in Example 8.1 of \cite{liu2008conley}. The method involves examining random dynamical systems within a sufficiently small neighbourhood of a deterministic system generated by a given map, ensuring that the continuation property holds in that neighbourhood. The Wazewski property is then applied to establish the existence of a non-trivial random invariant set for the perturbed systems. Lastly, the behaviour of elements within this random invariant set is analysed to verify that they satisfy the conditions required for an $(\epsilon,k)$-random periodic orbit as in Definition \ref{def:epsilonkorbitaale}. 

\begin{proof}[Proof of Theorem~\ref{thm:main_forcing}]
Fix $k\in\mathcal T$. We prove items (1)--(3) of the theorem for this $k$.

Let $S=\{p_1,\dots,p_k\}\subset \mathrm{int}\,C$ be a hyperbolic $k$--cycle of $f$
(of minimal period $k$), labelled so that $f(p_i)=p_{i+1}$ for $i<k$ and $f(p_k)=p_1$.
Such an orbit exists by the hypotheses of Theorem~\ref{thm:main_forcing}.
Set $g:=f^k$.

\medskip
\noindent\textbf{Step 1: deterministic isolating neighbourhood with forcing itinerary.}
By Theorem~\ref{GenMapsConleyIndex}, there exist pairwise disjoint compact intervals
\[
  N_i=[p_i-\varepsilon_i,p_i+\varepsilon_i]\subset \mathrm{int}\,C,\qquad i=1,\dots,k,
\]
such that, with $N:=\bigcup_{i=1}^k N_i$, the set $N$ is an isolating neighbourhood for $S$,
and $h(S)\neq 0$. Moreover, shrinking the $N_i$ if necessary we may (and do) assume that

\begin{enumerate}
\item[(i)] (\emph{itinerary forcing for $f$}) for all $i$ (indices mod $k$),
\[
  f(N_i)\subset \mathrm{int}(N_{i+1});
\]
\item[(ii)] (\emph{uniform hyperbolicity for $g$ on each $N_i$}) there exists $\beta>0$
such that for every $i$ either
\[
  \inf_{x\in N_i}|g'(x)|\ge 1+3\beta
  \qquad\text{or}\qquad
  \sup_{x\in N_i}|g'(x)|\le 1-3\beta;
\]
\item[(iii)] (\emph{separation margin}) writing
\[
  \eta := \min_{1\le i\le k} \dist\bigl(f(N_i),\, C\setminus \mathrm{int}(N_{i+1})\bigr) \;>\;0,
\]
we have a positive buffer $\eta$ between $f(N_i)$ and the complement of $\mathrm{int}(N_{i+1})$.
\end{enumerate}

\medskip
\noindent\textbf{Step 2: choosing $\epsilon$ and obtaining the random forced itinerary.}
Let $\epsilon>0$ be small enough so that for every $\varphi\in\mathcal R^1_\epsilon(f)$
(Definition~\ref{def:R1eps-simple}) we have:
\begin{enumerate}
\item[(a)] (\emph{one-step forcing persists}) since $\sup_{\omega,x\in C}|\varphi(\omega,x)-f(x)|<\epsilon$,
taking $\epsilon<\eta$ yields
\[
  \varphi(\omega,N_i)\subset \mathrm{int}(N_{i+1})
  \qquad\text{for all $\omega\in\Omega$ and all $i$};
\]
\item[(b)] (\emph{$C^1$ control along $N$}) on $N$ we have uniform $C^1$--closeness:
\[
  \sup_{\omega\in\Omega}\sup_{x\in N}
  \Bigl(|\varphi(\omega,x)-f(x)|+|\partial_x\varphi(\omega,x)-f'(x)|\Bigr)<\epsilon,
\]
and all breakpoints lie in $C\setminus N$.
\end{enumerate}

Let $\phi$ be the cocycle generated by $\varphi$ (as in \eqref{eq_RDS_GenByMap} with $\mathbb T=\mathbb N$).
Property (a) implies the \emph{forced itinerary} for $\phi$:
for every $\omega$ and every $x\in N_i$,
\[
  \phi(1,\omega,x)=\varphi(\omega,x)\in \mathrm{int}(N_{i+1}),
\]
hence inductively
\[
  \phi(j,\omega,x)\in \mathrm{int}(N_{i+j})
  \qquad\text{for }j=0,1,\dots,k-1
\]
(indices mod $k$), and in particular the $k$--step map
\[
  G_\omega(x):=\phi(k,\omega,x)
\]
satisfies $G_\omega(N_i)\subset \mathrm{int}(N_i)$ for every $i$ and every $\omega$.

Moreover, because all intermediate iterates stay in $N$ and all breakpoints lie in $C\setminus N$,
the composition defining $G_\omega$ is $C^1$ on each $N_i$. By the chain rule and (b), and since
$k$ is fixed and $f'$ is bounded on $N$, there exists a constant $C>0$ such that
\[
  \sup_{\omega\in\Omega}\sup_{x\in N}\bigl|(G_\omega)'(x)-g'(x)\bigr|\le C\epsilon.
\]

Choosing $\epsilon$ so small that $C\epsilon\le\beta$ and using Step~1(ii), we obtain the
\emph{uniform derivative dichotomy}: for each $i$, matching the deterministic alternative in
Step~1(ii), either
\[
  \inf_{\omega\in\Omega}\inf_{x\in N_i} |(G_\omega)'(x)| \ge 1+2\beta,
  \qquad\text{or}\qquad
  \sup_{\omega\in\Omega}\sup_{x\in N_i} |(G_\omega)'(x)| \le 1-2\beta.
\]

\medskip
\noindent\textbf{Step 3: existence and uniqueness of a random $k$--periodic point in $N_1$.}

\smallskip
\noindent\emph{Step 3(a): random isolated invariant set for $G$ in $N_1$.}
Let $g:=f^{k}$. Then $p_1$ is a hyperbolic fixed point of $g$, and $N_1$ is an
isolating neighbourhood for $\{p_1\}$ with non-trivial Conley index
(Theorem~\ref{GenMapsConleyIndex} applied to $g$ on $N_1$), hence
\[
  h\bigl(\{p_1\},g\bigr)\neq 0.
\]

Define the random map over the base $(\Omega,\theta^{k})$ by
\[
  G(\omega,x):=\phi(k,\omega,x).
\]

By the $C^{1}$--closeness chosen in Step~2 (shrinking $\varepsilon$ if necessary),
$N_1$ remains isolating for the $k$--step dynamics, i.e.
\[
  \Inv(N_1,G)(\omega)\subset \operatorname{int}(N_1)
  \qquad\text{for $\mathbb P$--a.e.\ $\omega$},
\]
so $N_1$ is a random isolating neighbourhood for $G$.

Set
\[
  A_1(\omega):=\Inv(N_1,G)(\omega).
\]

By the continuation property of the \emph{random cohomological} Conley index (applied to $G$
as a random perturbation of $g$ on $N_1$), we have that the Conley index is non trivial.

Consequently, by the random Wa\.zewski principle, $A_1(\omega)\neq\emptyset$ on
a $\theta^{k}$--invariant set of positive probability; in particular, if $\theta^{k}$
is ergodic then $A_1(\omega)\neq\emptyset$ for $\mathbb P$--a.e.\ $\omega$.

Finally, by definition of $\Inv(N_1,G)(\omega)$, for each $\omega$ with
$A_1(\omega)\neq\emptyset$ and each $x\in A_1(\omega)$ there exists a full
$G$--orbit $(x_n)_{n\in\mathbb Z}\subset N_1$ such that
\[
  x_0=x,\qquad x_{n+1}=G(\theta^{nk}\omega,x_n)\quad\text{for all }n\in\mathbb Z.
\]

\smallskip
\noindent\emph{Step 3(b): uniqueness in each fibre.}
Fix $\omega$ and suppose $p,q\in A_1(\omega)$. Choose full $G$--orbits
$(p_n)_{n\in\mathbb Z}$, $(q_n)_{n\in\mathbb Z}$ in $N_1$ with $p_0=p$, $q_0=q$ and
\[
  p_{n+1}=G_{\theta^{nk}\omega}(p_n),\quad
  q_{n+1}=G_{\theta^{nk}\omega}(q_n)\qquad (n\in\mathbb Z).
\]

Set $d_n:=|p_n-q_n|$.

Assume first that the expanding alternative holds on $N_1$, i.e.
\[
  \inf_{\omega'\in\Omega}\inf_{x\in N_1}|(G_{\omega'})'(x)|\ge 1+2\beta.
\]

Then, by the mean value theorem applied to $G_{\theta^{nk}\omega}$ on $N_1$,
\[
  d_{n+1}
  =|G_{\theta^{nk}\omega}(p_n)-G_{\theta^{nk}\omega}(q_n)|
  \ge (1+2\beta)\,d_n\qquad (n\ge 0).
\]

Hence
\[
  d_n\ge (1+2\beta)^n d_0\qquad (n\ge 0).
\]

Since $p_n,q_n\in N_1$ for all $n$, we have $d_n\le \diam(N_1)$, so this is possible
only if $d_0=0$, i.e.\ $p=q$.

Assume now that the contracting alternative holds on $N_1$, i.e.
\[
  \sup_{\omega'\in\Omega}\sup_{x\in N_1}|(G_{\omega'})'(x)|\le 1-2\beta.
\]

Then for every $n\in\mathbb Z$,
\[
  d_{n+1}
  =|G_{\theta^{nk}\omega}(p_n)-G_{\theta^{nk}\omega}(q_n)|
  \le (1-2\beta)\,d_n.
\]

Applying this inequality for $n=-1,-2,\dots,-m$ yields
\[
  d_0 \le (1-2\beta)^m\,d_{-m}\le (1-2\beta)^m\,\diam(N_1).
\]

Letting $m\to\infty$ gives $d_0=0$, hence again $p=q$.

Therefore $A_1(\omega)$ is a singleton whenever it is non-empty, and so for almost
every $\omega$ there exists a unique point in $A_1(\omega)$.

\smallskip
\noindent\emph{Step 3(c): construction of the random periodic point and minimality of the period.}
Since $A_1$ is a random compact set and each $A_1(\omega)$ is a singleton for almost every $\omega$,
there exists a measurable map $x:\Omega\to N_1$ such that $x(\omega)\in A_1(\omega)$ for
$\mathbb P$--a.e.\ $\omega$.

By invariance of $\Inv(N_1,G)(\omega)$ under $G$, for each $\omega$ we have
\[
  G_\omega(x(\omega))\in A_1(\theta^k\omega).
\]

By uniqueness in each fibre, this forces
\[
  G_\omega(x(\omega))=x(\theta^k\omega),
\]
i.e.
\[
  \phi(k,\omega,x(\omega))=x(\theta^k\omega).
\]
Thus $x$ is a random periodic point of period $k$ in the sense of
Definition~\ref{def_randomPeriodicPoint}.

To see that this period is minimal, note that by the forced itinerary in Step~2,
for every $\omega$ and every $\ell=1,\dots,k-1$,
\[
  \phi(\ell,\omega,x(\omega))\in N_{\ell+1},
\]
while $x(\theta^\ell\omega)\in N_1$. Since the $N_i$ are pairwise disjoint,
\[
  \phi(\ell,\omega,x(\omega))\neq x(\theta^\ell\omega)
  \qquad\text{for all $\omega$ and all $\ell=1,\dots,k-1$}.
\]

Hence $x$ has minimal random period $k$ in the sense of Definition~\ref{def_randomPeriodicPoint}.

This proves item (1) for the chosen $k$.

\medskip
\noindent\textbf{Step 4: random periodic orbit (set-valued) and minimality under ergodicity.}
Assume now that $\theta^k$ is ergodic.
From the random periodic point $x$ of minimal random period $k$, define the random set
\[
  P(\omega)
  :=\Bigl\{\phi(j,\theta^{-j}\omega,x(\theta^{-j}\omega)):\; j=0,1,\dots,k-1\Bigr\}.
\]

Then $P(\omega)\subset N_1\cup\cdots\cup N_k$, and by the forced itinerary the $k$ points
lie one in each $N_i$, hence $\#P(\omega)=k$ for all $\omega$.

We first check that $P$ is random invariant. Fix $\omega$ and take any $z\in P(\omega)$.
Then there exists $j\in\{0,\dots,k-1\}$ such that
\[
  z=\phi(j,\theta^{-j}\omega,x(\theta^{-j}\omega)).
\]

Using the cocycle property,
\[
  \phi(1,\omega,z)
  =\phi\bigl(1,\omega,\phi(j,\theta^{-j}\omega,x(\theta^{-j}\omega))\bigr)
  =\phi\bigl(j+1,\theta^{-j}\omega,x(\theta^{-j}\omega)\bigr).
\]
If $j\le k-2$, set $\tilde j=j+1\in\{1,\dots,k-1\}$. Since
$\theta^{-j}\omega=\theta^{-\tilde j}(\theta\omega)$, we obtain
\[
  \phi(1,\omega,z)
  =\phi\bigl(\tilde j,\theta^{-\tilde j}(\theta\omega),
               x(\theta^{-\tilde j}(\theta\omega))\bigr)\in P(\theta\omega).
\]

If $j=k-1$, then
\[
  \phi(1,\omega,z)
  =\phi\bigl(k,\theta^{-(k-1)}\omega,x(\theta^{-(k-1)}\omega)\bigr).
\]

By the random periodicity of $x$ (Step~3(c)), applied at $\theta^{-(k-1)}\omega$,
\[
  \phi\bigl(k,\theta^{-(k-1)}\omega,x(\theta^{-(k-1)}\omega)\bigr)
  =x\bigl(\theta^k(\theta^{-(k-1)}\omega)\bigr)=x(\theta\omega),
\]

and $x(\theta\omega)\in P(\theta\omega)$ corresponds to the term $j=0$ in the
definition of $P(\theta\omega)$. Hence
\[
  \phi(1,\omega,P(\omega))\subset P(\theta\omega)
  \qquad\text{for all }\omega.
\]

Conversely, fix $\omega$ and take any $w\in P(\theta\omega)$. Then there exists
$j\in\{0,\dots,k-1\}$ such that
\[
  w=\phi\bigl(j,\theta^{-j}(\theta\omega),
                x(\theta^{-j}(\theta\omega))\bigr).
\]

If $j\ge 1$, set $j':=j-1\in\{0,\dots,k-2\}$ and define
\[
  z:=\phi\bigl(j',\theta^{-j'}\omega,x(\theta^{-j'}\omega)\bigr)\in P(\omega).
\]

Using again the cocycle property and the identity
$\theta^{-j}(\theta\omega)=\theta^{-j'}\omega$, we get
\[
  \phi(1,\omega,z)
  =\phi\bigl(j,\theta^{-j'}\omega,x(\theta^{-j'}\omega)\bigr)
  =\phi\bigl(j,\theta^{-j}(\theta\omega),
              x(\theta^{-j}(\theta\omega))\bigr)
  =w.
\]

If $j=0$, then $w=x(\theta\omega)$. By the random periodicity of $x$,
applied at $\theta^{-(k-1)}\omega$, we have
\[
  x(\theta\omega)
  =\phi\bigl(k,\theta^{-(k-1)}\omega,x(\theta^{-(k-1)}\omega)\bigr).
\]

Set
\[
  z:=\phi\bigl(k-1,\theta^{-(k-1)}\omega,x(\theta^{-(k-1)}\omega)\bigr)\in P(\omega).
\]

Then, by the cocycle property,
\[
  \phi(1,\omega,z)
  =\phi\bigl(k,\theta^{-(k-1)}\omega,x(\theta^{-(k-1)}\omega)\bigr)
  =x(\theta\omega)=w.
\]
Thus, in all cases $w=\phi(1,\omega,z)$ for some $z\in P(\omega)$, and hence
\[
  P(\theta\omega)\subset \phi(1,\omega,P(\omega))
  \qquad\text{for all }\omega.
\]

Combining the two inclusions, we conclude that
\[
  \varphi(\omega)P(\omega)
  =\phi(1,\omega,P(\omega))
  =P(\theta\omega)
  \qquad\text{for all }\omega.
\]

Therefore $P$ is a random periodic orbit of period $k$ in the sense of
Definition~\ref{def_randomPeriodicOrbit}.

To see minimality in the sense of Definition~\ref{def_randomPeriodicOrbit},
let $Q(\omega)\subset P(\omega)$ be a non-empty random invariant subset.
For each $i\in\{1,\dots,k\}$, denote by $y_i(\omega)$ the unique point of
$P(\omega)\cap N_i$ (existence and uniqueness follow from the construction of $P$).
Define the indicator vector
\[
  v(\omega)=(v_1(\omega),\dots,v_k(\omega))\in\{0,1\}^k,\qquad
  v_i(\omega):=\mathbf 1_{\{y_i(\omega)\in Q(\omega)\}}.
\]

By the forced itinerary and the definition of $P$, we have
\[
  \phi(1,\omega,y_i(\omega))=y_{i+1}(\theta\omega),\qquad i=1,\dots,k
\]
(with indices taken modulo $k$). Since $Q$ is invariant and contained in $P$, we obtain
for each $i$ and $\omega$:
\[
  y_i(\omega)\in Q(\omega) \quad\Longleftrightarrow\quad
  y_{i+1}(\theta\omega)\in Q(\theta\omega).
\]
Equivalently,
\[
  v_{i+1}(\theta\omega)=v_i(\omega),\qquad i=1,\dots,k,
\]
so that
\[
  v(\theta\omega)=\sigma(v(\omega)),
\]
where $\sigma$ is the cyclic shift on $\{0,1\}^k$.

Iterating $k$ times yields
\[
  v(\theta^k\omega)=\sigma^k(v(\omega))=v(\omega)
  \qquad\text{for all }\omega,
\]
so $v$ is $\theta^k$--invariant. By ergodicity of $\theta^k$, $v(\omega)$ is almost
surely constant, say $v(\omega)\equiv v^\ast$. The relation
$v(\theta\omega)=\sigma(v(\omega))$ then implies $v^\ast=\sigma(v^\ast)$, so $v^\ast$
is a fixed point of the cyclic shift. The only fixed points of $\sigma$ in $\{0,1\}^k$
are $(0,\dots,0)$ and $(1,\dots,1)$.

Since $Q$ is non-empty, we cannot have $v^\ast=(0,\dots,0)$, hence
$v^\ast=(1,\dots,1)$ and therefore $Q(\omega)=P(\omega)$ for almost every $\omega$.
Thus $P$ has minimal period $k$ in the sense of
Definition~\ref{def_randomPeriodicOrbit}.

This proves item (2) for the chosen $k$.

\medskip
\noindent\textbf{Step 5: existence of a $(\delta,k)$--random periodic orbit.}
Choose $\delta>0$ such that
\[
  \max_{1\le i\le k}\diam(N_i) < \delta
  \quad\text{and}\quad
  \min_{i\neq j}\dist(N_i,N_j) > \delta,
\]
which is possible since the $N_i$ are pairwise disjoint and may be chosen arbitrarily small.

For each $\ell=0,\dots,k-1$, the fibre
\[
  S_\ell(\omega)=\{\phi(\ell+mk,\omega,x(\omega)):\; m\in\mathbb N_0\}
\]
satisfies $S_\ell(\omega)\subset N_{\ell+1}$ by the forced itinerary, hence
$\diam(S_\ell(\omega))\le \diam(N_{\ell+1})<\delta$.
If $i\neq j$, then $S_i(\omega)\subset N_{i+1}$ and $S_j(\omega)\subset N_{j+1}$,
so $\dist(S_i(\omega),S_j(\omega))\ge \dist(N_{i+1},N_{j+1})>\delta$.
Therefore $x$ generates a $(\delta,k)$--random periodic orbit in the sense of
Definition~\ref{def:epsilonkorbitaale}.

This proves item (3) for the chosen $k$.

\medskip
Since $k\in\mathcal T$ was arbitrary, the theorem follows.
\end{proof}

\subsection{Proof of Theorem \ref{thm:SharkRealisationAle} }

 \begin{definition}[Truncated tent maps]\label{def:truncated-tent}
Let $T\colon [0,1]\to[0,1]$ be the standard tent map
\[
  T(x) =
  \begin{cases}
    2x,         & 0 \le x \le \tfrac12,\\[0.3em]
    2 - 2x,     & \tfrac12 \le x \le 1.
  \end{cases}
\]
For each $h\in(0,1]$ we define the \emph{truncated tent map} $T_h\colon[0,1]\to[0,1]$ by
\[
  T_h(x) := \min\{T(x),h\}.
\]
Equivalently,
\[
  T_h(x) =
  \begin{cases}
    2x,         & 0 \le x \le \tfrac{h}{2},\\[0.3em]
    h,          & \tfrac{h}{2} \le x \le 1 - \tfrac{h}{2},\\[0.3em]
    2 - 2x,     & 1 - \tfrac{h}{2} \le x \le 1.
  \end{cases}
\]
\end{definition}
\begin{figure}[ht]
  \centering
  \begin{tikzpicture}[scale=3.5, >=latex]
    \def\h{0.6}

    \draw[thick] (0,0) rectangle (1,1);

    \draw[gray!50] (0,0) -- (1,1);

    \draw[gray!50] (0,0) -- (0.5,1) -- (1,0);

    \draw[gray!50] (0,\h) -- (1,\h);

    \draw[thick]
      (0,0) --
      ({\h/2},\h) --
      ({1 - \h/2},\h) --
      (1,0);

    \node[left] at (0,\h) {$h$};

  
    \node[left] at (-0.15,0.25) {$T_h$};
  \end{tikzpicture}
  \caption{The truncated tent map $T_h$: the thick graph is $T_h$, 
  while the gray graph is the standard tent map and the diagonal.}
  \label{fig:truncated-tent}
\end{figure}
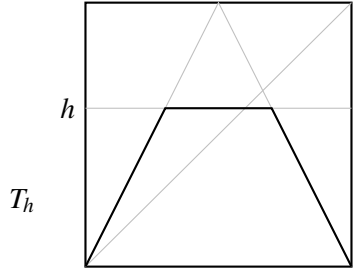

\begin{thm}[Truncated tent maps and finite Sharkovskiĭ tails]
\label{thm:truncated-tent-tail}

Let $T_h\colon[0,1]\to[0,1]$ be the truncated tent maps from
Definition~\ref{def:truncated-tent}.
For each $m\in\mathbb{N}$ with $m\neq 1$ there exists a critical parameter
$h(m)\in(0,1]$ such that the set of all minimal periods of $T_{h(m)}$ coincides
with the Sharkovskiĭ tail generated by $m$, that is
\[
\{k\in\mathbb N : k = m \text{ or } k \prec m\}.
\]

Moreover, if $m^+$ denotes the immediate successor of $m$ in the Sharkovskiĭ
order, then there exists a parameter $\tilde h(m)\in\bigl(h(m),h(m^+)\bigr)$ such
that $T_{\tilde h(m)}$ has the same set of minimal periods, admits an $m$--cycle
contained in $\mathrm{int}\,C=(0,1)$, and satisfies $T_{\tilde h(m)}\in\mathcal H_C$.
In particular, for each $k\prec  m$ with $k>1$ there exists a $k$--cycle
$\mathcal{O}_k$ of $T_{\tilde h(m)}$ such that $\mathcal{O}_k\subset(0,1)$, the
orbit $\mathcal{O}_k$ avoids the plateau
\[
  \{x\in[0,1]:T_{\tilde h(m)}(x)=\tilde h(m)\}
\]
and its endpoints $\tilde h(m)/2$ and $1-\tilde h(m)/2$, and it also avoids the
boundary points $0$ and $1$.
\end{thm}

\begin{proof}
Consider first the full tent map $T_1(x)=1-2\bigl|x-\tfrac12\bigr|$ on $[0,1]$.
As shown in Burns--Hasselblatt~\cite{BurnsHasselblatt}, the associated family of
truncated tent maps
\[
  T_h(x)=\min\{h,T_1(x)\}, \qquad 0\le h\le 1,
\]
has the following properties.
The map $T_0$ has only the fixed point $0$, while $T_1$ has a $3$--cycle and hence,
by the Sharkovskiĭ forcing theorem, possesses periodic orbits of all natural
minimal periods. Moreover, $T_1$ has only finitely many periodic points of each
given minimal period. Finally, if $h\le k$, then any periodic orbit
$O\subset[0,h)$ of $T_h$ is also a periodic orbit of $T_k$, and conversely any
periodic orbit $O\subset[0,h]$ of $T_k$ is a periodic orbit of $T_h$;
see~\cite[Sec.~5]{BurnsHasselblatt} for details.

Note that $T_h(0)=0$ for every $h\in[0,1]$, so $\{0\}$ is always a fixed point of $T_h$.

For each $m\in\mathbb N$, define the critical height
\[
  h(m):=\min\{\max O: O \text{ is an $m$--cycle of }T_1\}.
\]
Since $T_1$ has only finitely many $m$--cycles, this minimum is well defined.
Using the monotonic dependence on the parameter and the specific structure of
the family $T_h$, Burns and Hasselblatt show the following (see again
\cite[Sec.~5]{BurnsHasselblatt}): for every $h\in[0,1]$ and every $l\in\mathbb N$,
the map $T_h$ has an $l$--cycle contained in $[0,h)$ if and only if $h(l)<h$.
In addition, the forward orbit of the point $h(m)$ under $T_{h(m)}$ is an
$m$--cycle, and every periodic orbit of $T_{h(m)}$ distinct from $\orb(h(m))$
is contained in $[0,h(m))$. Finally, the ordering of the heights $h(\cdot)$
coincides with the Sharkovskiĭ ordering in the sense that $h(l)<h(m)$
holds if and only if $l\prec  m$.

Fix $m\in\mathbb N$ with $m\neq 1$. The description of the heights yields the set
of minimal periods of $T_{h(m)}$. Indeed, the orbit $\orb(h(m))$ gives a periodic
orbit of minimal period $m$, so $m$ occurs as a period. If $k\prec  m$, then
$h(k)<h(m)$, and hence (taking $h=h(m)$ and $l=k$ in the characterization above)
there exists a $k$--cycle of $T_{h(m)}$ contained in $[0,h(m))$.
Conversely, let $l$ be the minimal period of any periodic orbit $O$ of $T_{h(m)}$
which is not a fixed point. If $\max O<h(m)$, then the same characterization
implies $h(l)<h(m)$ and hence $l\prec  m$. If $\max O=h(m)$, then by
construction $O=\orb(h(m))$ and therefore $l=m$. It follows that the set of all
minimal periods of $T_{h(m)}$ is exactly the tail
\[
\{k\in\mathbb N: k=m \text{ or } k\prec m\}.
\]

We now choose a parameter slightly above $h(m)$ in order to obtain a map in
$\mathcal H_C$ while keeping the same set of minimal periods. Let $m^+$ denote the
immediate successor of $m$ in the Sharkovskiĭ order, and choose a parameter
\[
  \tilde h(m)\in\bigl(h(m),h(m^+)\bigr).
\]
Then for every $\ell\in\mathbb N$ we have $h(\ell)<\tilde h(m)$ if and only if
$h(\ell)\le h(m)$, because $\tilde h(m)$ lies strictly between $h(m)$ and the next
height $h(m^+)$. Hence $T_{\tilde h(m)}$ has an $\ell$--cycle contained in
$[0,\tilde h(m))$ if and only if $\ell\prec  m$. In particular, the set
of minimal periods of $T_{\tilde h(m)}$ is the the same tail
$\{k\in\mathbb N: k=m \text{ or } k\prec m\}$.

Moreover, since $h(m)<\tilde h(m)$, the characterization implies that
$T_{\tilde h(m)}$ admits an $m$--cycle $O_m$ contained in $[0,\tilde h(m))$.
For each $k\prec  m$ with $k>1$, choose a $k$--cycle $\mathcal O_k$ of
$T_{\tilde h(m)}$ contained in $[0,\tilde h(m))$, whose existence follows from
the same characterization. Since $[0,\tilde h(m))\subset[0,1)$, no point of
$\mathcal O_k$ can be equal to $1$. Also $\mathcal O_k$ cannot contain $0$
because $k>1$ and $0$ is a fixed point. Hence $\mathcal O_k\subset(0,1)$.

Let
\[
  P_{\tilde h(m)}:=\{x\in[0,1]:T_{\tilde h(m)}(x)=\tilde h(m)\}
  =\bigl[\tilde h(m)/2,\,1-\tilde h(m)/2\bigr]
\]
be the plateau of $T_{\tilde h(m)}$. If $x\in \mathcal O_k\cap P_{\tilde h(m)}$,
then $T_{\tilde h(m)}(x)=\tilde h(m)$, so the next point of the orbit is
$\tilde h(m)\notin[0,\tilde h(m))$, contradicting $\mathcal O_k\subset[0,\tilde h(m))$.
Thus $\mathcal O_k$ is disjoint from $P_{\tilde h(m)}$, and in particular avoids
its endpoints $\tilde h(m)/2$ and $1-\tilde h(m)/2$, which are the corner points of
the truncated tent. This proves the claimed geometric properties.

Finally, we show that $T_{\tilde h(m)}\in\mathcal H_C$ in the sense of
Definition~\ref{def:whatever-HC-is}. Recall that $T_{\tilde h(m)}$ is $C^1$ on
$[0,1]$ except at the two corner points
\[
  c^-:=\tilde h(m)/2, \qquad c^+:=1-\tilde h(m)/2.
\]
We further choose $\tilde h(m)\in(h(m),h(m^+))$ so that none of the points
$\tilde h(m)$, $c^-$, and $c^+$ is periodic for $T_{\tilde h(m)}$.
This is possible because, for each $r\in\mathbb N$, each of the equations
$T_h^r(h)=h$, $T_h^r(c^-)=c^-$, and $T_h^r(c^+)=c^+$ has only finitely many
solutions in $h\in(0,1]$, and hence the set of parameters for which one of
these points is periodic is countable.

With this choice, no periodic orbit of $T_{\tilde h(m)}$ can meet the plateau
$P_{\tilde h(m)}$, because every point of $P_{\tilde h(m)}$ maps in one step to
$\tilde h(m)$, and $\tilde h(m)$ is not periodic. Therefore every periodic orbit
$O$ which is not a fixed point is disjoint from $P_{\tilde h(m)}$, hence avoids
the corner points $\{c^-,c^+\}$ and lies in $\mathrm{int}\,C=(0,1)$.
Consequently, each such orbit admits a neighbourhood $U$ on which
$T_{\tilde h(m)}$ is $C^1$.

Let $O$ be a periodic orbit of $T_{\tilde h(m)}$ which is not a fixed point, and
let $k$ be its minimal period. Since $O$ is disjoint from the plateau, every point
of $O$ lies on one of the linear branches of $T_{\tilde h(m)}$, where
$T'_{\tilde h(m)}(x)\in\{2,-2\}$. Hence for all $p\in O$ we have
\[
  \bigl|(T_{\tilde h(m)}^{k})'(p)\bigr|=2^{k}\neq 1.
\]
Thus every non fixed point periodic orbit of $T_{\tilde h(m)}$ is hyperbolic and
lies in $\mathrm{int}\,C$, hence $T_{\tilde h(m)}\in\mathcal H_C$.
\end{proof}

\begin{proof}[Proof of Theorem~\ref{thm:SharkRealisationAle}]
Let $\mathcal{T}$ be a finite Sharkovskiĭ's tail and denote by $n$ the
first element of $\mathcal{T}$ in Sharkovskiĭ's order. Thus
$\mathcal{T}\subset \mathcal{T}(n)$ and $\mathcal{T}$ is a finite initial
segment of this tail.

\medskip\noindent

Consider the family of truncated tent maps from
Definition~\ref{def:truncated-tent}. 
By Theorem~\ref{thm:truncated-tent-tail} (with $m=n$), there exists $h\in(0,1]$ such that
the interval map $f := T_h \colon C\to C$, with $C=[0,1]$, has minimal periods given
exactly by the tail $\mathcal{T}(n)$.
In particular, $f$ admits a periodic orbit of minimal period $n$ contained in $\mathrm{int}\,C$.
Combining Theorem~\ref{thm:truncated-tent-tail} with the discussion above, we also have
$f\in\mathcal{H}_{\mathcal{T}}$.

\medskip\noindent
Set $\mathcal{O}_n := \mathcal{T}\subset\mathcal{T}(n)$. Since $f$ has
a hyperbolic periodic orbit of minimal period $n$ in $\mathrm{int}\,C$,
we may apply Theorem~\ref{thm:main_forcing} to the pair $(f,\mathcal{O}_n)$.

We obtain $\epsilon>0$ such that, for every RDS $\phi$ generated by
some $\varphi\in\mathcal{R}^1_{\epsilon}(f)$ with noise $\theta$
ergodic over $\mathbb{P}$, the set of minimal periods of random periodic
points of $\phi$ contains $\mathcal{T}$. If, in addition, $\theta^k$ is
ergodic for all $k\in\mathbb{N}$, then the set of minimal periods of
random periodic orbits of $\phi$ also contains $\mathcal{T}$. Moreover,
there exists $\delta>0$ such that, for the same choice of $\phi$, the
set of minimal periods of $(\delta,k)$--random periodic orbits contains
$\mathcal{T}$ as well.

\medskip

In particular, for any such choice of $\varphi$ and $\theta$, the set of minimal periods
of random periodic points of $\phi$ contains $\mathcal{T}$, and the same holds for random periodic
orbits and for $(\delta,k)$--random periodic orbits.
This is exactly what is claimed in the three assertions of
Theorem~\ref{thm:SharkRealisationAle}. \qedhere

\end{proof}

\section{Example of random perturbations} \label{Examples}
\subsection{Random tent map}

We illustrate Theorem \ref{thm:main_forcing} using a random perturbation of the tent map. We construct an RDS by selecting maps from a family of asymmetric tent maps with displaced peaks, as shown in Figure \ref{fig:random-tent-family}.

\pgfmathdeclarefunction{Tent}{2}{%
  \pgfmathparse{(#1 <= 0.5 + #2) ?
    2*#1/(1 + 2*#2) :
    2*(1-#1)/(1 - 2*#2)}%
}

\begin{figure}[ht]
  \centering
  \begin{tikzpicture}
    \begin{axis}[
      width=8cm, height=6cm,
      axis lines=left,
      xlabel={$x_n$},
      ylabel={$x_{n+1}$},
      xmin=0, xmax=1,
      ymin=0, ymax=1,
      grid=major,
      legend style={at={(1.02,0.5)},anchor=west},
      legend cell align={left},
    ]

      \addplot[gray!60, dashed, domain=0:1, samples=100] {x};
      \addlegendentry{$y=x$}

      \pgfmathsetmacro{\ga}{-0.08}
      \pgfmathsetmacro{\gb}{0.02}
      \pgfmathsetmacro{\gc}{0.12}

      \addplot[plotblue2, very thick, domain=0:1, samples=200]
        {Tent(x,\ga)};
      \addlegendentry{$T_{\gamma_1}$}

      \addplot[plotred, very thick, domain=0:1, samples=200]
        {Tent(x,\gb)};
      \addlegendentry{$T_{\gamma_2}$}

      \addplot[plotgreen, very thick, domain=0:1, samples=200]
        {Tent(x,\gc)};
      \addlegendentry{$T_{\gamma_3}$}


      \pgfmathsetmacro{\xzero}{0.15}

      \pgfmathsetmacro{\xone}{Tent(\xzero,\ga)}
      \pgfmathsetmacro{\xtwo}{Tent(\xone,\gb)}
      \pgfmathsetmacro{\xthree}{Tent(\xtwo,\gc)}
      \pgfmathsetmacro{\xfour}{Tent(\xthree,\ga)}

      \draw[red, ->, thick]
        (axis cs:\xzero,\xzero) --
        (axis cs:\xzero,\xone) --
        (axis cs:\xone,\xone);

      \draw[red, ->, thick]
        (axis cs:\xone,\xone) --
        (axis cs:\xone,\xtwo) --
        (axis cs:\xtwo,\xtwo);

      \draw[red, ->, thick]
        (axis cs:\xtwo,\xtwo) --
        (axis cs:\xtwo,\xthree) --
        (axis cs:\xthree,\xthree);

      \draw[red, ->, thick]
        (axis cs:\xthree,\xthree) --
        (axis cs:\xthree,\xfour) --
        (axis cs:\xfour,\xfour);

      \addplot[only marks, mark=*, mark size=1.3pt, red]
        coordinates {
          (\xzero,\xone)
          (\xone,\xtwo)
          (\xtwo,\xthree)
          (\xthree,\xfour)
        };

      \node[below, red] at (axis cs:\xzero,0) {$x_0$};

      \addlegendimage{red, thick, -{Latex[length=3pt]}}
      \addlegendentry{Random trajectory}
    \end{axis}
  \end{tikzpicture}
  \caption{Random iteration of asymmetric tent maps $T_\gamma$ with common peak
    height~$1$. The coloured graphs represent several maps
    $T_\gamma$ with different peak locations $x = \tfrac{1}{2} + \gamma$,
    and the red staircase shows a sample trajectory in the $(x_n,x_{n+1})$ plane
    obtained by randomly switching between these maps.}
  \label{fig:random-tent-family}
\end{figure}

Formally, consider the space $X=(-\xi, \xi)$ with $\xi<\frac{1}{2}$ fixed with the Borel $\sigma$-field and equipped with the uniform distribution over this space. This parameter $\xi$ will be our source of control over the perturbation that we will add onto the system. We associate to each element $\gamma \in X$ a function of the family of maps $\mathcal{F}$ with elements $T(\gamma):[0,1]\rightarrow[0,1]$ of the form:
\begin{equation}
    T(\gamma)(x)=\begin{cases}
        \frac{2x}{1+2\gamma},\, \text{ if }0\leq x\leq \frac{1}{2}+\gamma \\
        \frac{2(1-x)}{1-2\gamma},\, \text{ if }\frac{1}{2}+\gamma\leq x\leq 1
    \end{cases}
\end{equation}

Those functions are called asymmetric tent maps. Therefore, we have a probability measure over the space $\mathcal{F}$. Consider the space of sequences with entries in $\mathcal{F}$ given by:
\begin{equation}
    \Omega = \mathcal{F}^{\mathbb{N}}=\{\omega\,:\,(T(\gamma_0), T(\gamma_1),\dots)\text{ where }T(\gamma_k)\in \mathcal{F}\text{ for any }k\in \mathbb{N}\}
\end{equation}
and we equip $\Omega$ with the Bernoulli measure defined on the cylinders induced by the uniform measure on $X$. Therefore, we have the random dynamical system given by:

\begin{equation} \label{defTendaAle}
    \phi(n,\omega)(\cdot)=T(\gamma_{n-1})\circ\dots\circ T(\gamma_0)(\cdot)
\end{equation}

Essentially, on each iteration of the system we randomly select an asymmetric tent map and apply this map to $[0,1]$, so that we are moving those points in time on each step by a (possibly, and very likely) different asymmetric tent map as shown in the Eq. \eqref{defTendaAle}. 

We will use the Continuation Property stated in Theorem \ref{ContinuationProperty} in order to decide about the existence of a maximal invariant set inside a neighbourhood of the point $\frac{2}{3}$. Take $\epsilon>0$ satisfying $[\frac{2}{3}-2\epsilon, \frac{2}{3}+2\epsilon]\subset(\frac{1}{2},1)$, i.e., the properties we have demanded in Subsection \ref{GenMapsConleyIndex} and now suppose that $|\xi|\frac{\epsilon}{4}$ and demand that $\frac{1}{2} + |\xi|\leq \frac{2}{3}-2\epsilon$ (this second assumption guarantees that the map is decreasing in $N$). Define $N = [\frac{2}{3}-\epsilon,\frac{2}{3}+\epsilon]$ and, for $\lambda\in [0,1]$, consider the family of random maps $T_{\lambda}$ defined from $[0,1]$ onto itself given by:

\begin{equation}
    T_{\lambda}(\gamma)(x) = \begin{cases}
        \frac{2x}{1+2\lambda\gamma},\, \text{ if }0\leq x \leq \frac{1}{2}+\lambda\gamma \\
        \frac{2(1-x)}{1-2\lambda\gamma},\,\text{ if }\frac{1}{2}+\lambda\gamma \leq x \leq 1
    \end{cases}
\end{equation}
and note that when $\lambda = 0$ we have the original tent map. 

We want to show that $N$ is a random isolating neighbourhood for all $T_{\lambda}(\cdot)$ for whichever choice of $\lambda\in [0,1]$. Once more, in order to prove this, we must check that $\text{Inv}(N)\subset\text{int}(N)$; to do so, it is enough to show that $\partial N=\{\frac{2}{3}-\epsilon, \frac{2}{3}+\epsilon\}$ is such that $T_{\lambda}(\gamma)(\partial N)\cap \text{int}N = \emptyset$ for all $\gamma$.

Let $\gamma\in X$ be given then. We have that $T_{\lambda}(\gamma)(\frac{2}{3}-\epsilon)= \frac{\frac{2}{3}+2\epsilon}{1-2\gamma}$ and we will show that $T_{\lambda}(\gamma)(\frac{2}{3}-\epsilon)>\frac{2}{3}+\epsilon$; indeed, $\frac{\frac{2}{3}+2\epsilon}{1-2\gamma}>\frac{2}{3}+2\epsilon\iff\frac{2}{3}+2\epsilon > (\frac{2}{3}+\epsilon)(1-2\gamma)$ since $\gamma < \frac{1}{2}$ by assumption. Therefore, we have that $\frac{2}{3}+2\epsilon > (\frac{2}{3}+\epsilon)(1-2\gamma)\iff\epsilon>\frac{-4\gamma}{3}-2\gamma\epsilon$ and, at last, since $\gamma$ can take positive values as well as negative ones, we have that $\epsilon > \frac{-4\gamma}{3}-2\gamma\epsilon$ would follow if we showed that $|\frac{-4\gamma}{3}-2\gamma\epsilon|<\epsilon$. But note that:

\begin{equation}
    |\frac{-4\gamma}{3}-2\gamma\epsilon| = |\gamma(\frac{-4}{3}-2\epsilon)| = |\gamma|(\frac{4}{3}+2\epsilon)<\frac{\epsilon}{4}(\frac{4}{3}+2\epsilon) = \frac{\epsilon}{3} + \frac{\epsilon^2}{2}<\frac{\epsilon}{3} + \frac{\epsilon}{2}<\epsilon
\end{equation}

On the other hand, we have that $T_{\lambda}(\gamma)(\frac{2}{3}+\epsilon) = \frac{\frac{2}{3}-2\epsilon}{1-2\gamma}$ and we will show that $T_{\lambda}(\gamma)(\frac{2}{3}+\epsilon)<\frac{2}{3}-\epsilon$; indeed, once more we have that $\gamma<\frac{1}{2}$ implies that $\frac{\frac{2}{3}-2\epsilon}{1-2\gamma}<\frac{2}{3}-\epsilon \iff \frac{2}{3}-2\epsilon < (\frac{2}{3}-\epsilon)(1-2\gamma)$ and if we expand the right hand side of this inequality we actually need to check that $-\epsilon < \frac{-4\gamma}{3}+2\gamma\epsilon\iff \epsilon>\frac{4\gamma}{3}-2\gamma\epsilon$ and, again, note that to do so it is enough to check that $|\frac{4\gamma}{3}-2\gamma\epsilon|<\epsilon$. But we have:

\begin{equation}
    |\frac{4\gamma}{3}-2\gamma\epsilon|= |\gamma(\frac{4}{3}-2\epsilon)|=|\gamma||(\frac{4}{3}-2\epsilon)|<|\gamma|(\frac{4}{3}+2\epsilon)<\frac{\epsilon}{4}(\frac{4}{3}+2\epsilon)<\epsilon
\end{equation}

Therefore $N$ is a random isolating neighbourhood for all maps $T_{\lambda}(\cdot)$ with $\lambda\in [0,1]$; using the Continuation Property from Theorem \ref{ContinuationProperty}, apart from a choice of $\xi$ possibly smaller, we have that $h(S_{\lambda}, T_{\lambda}(\cdot)) = h(S_0,T_0)\neq 0$ from the calculation done in Subsection \ref{GenMapsConleyIndex}. Therefore, using Wazewiski property from Theorem \ref{wazeProperty}, we have that the maximal invariant set of $N$ under the action of the random dynamical system $\varphi(n,\omega)(\cdot)$ defined as in Eq. \eqref{defTendaAle} is non empty. 

Note that the Continuation Property can be used in this case precisely because, since we are considering a family of functions defined on a compact set and taking values on this same set, the $C^0$-topology coincides with the topology generated by the infinite norm, i.e., for any function $f$ from $C$ to $C$ one has that $||f||_{\infty} = \max_{x\in C}|f(x)|$

\subsection{Random logistic map}
In this subsection we briefly indicate how our abstract results specialise to a random logistic family, and we omit the technical details, which follow directly from the general theory developed in the preceding sections. Let 
\[
  f_c(x) = c\,x(1-x), \qquad x \in [0,1],\; c \in (0,4],
\]
and define a random dynamical system by allowing the parameter \(c\) to vary
in time in a random way. Let \((\Omega,\mathcal{F},\mathbb{P},\theta)\) be a
metric dynamical system and let \(c\colon\Omega\to C \subset (0,4]\) be a
measurable map representing the random parameter; for instance, \(c(\omega)\)
may be distributed according to the normalised Lebesgue measure on a subset
\(C\). The associated random logistic map is the RDS \(\varphi\) defined by
\[
  \varphi(n,\omega,x) 
    = f_{c(\theta^{n-1}\omega)}\circ\cdots\circ f_{c(\omega)}(x), 
    \qquad x\in[0,1],\; n\in\mathbb{N}.
\]
In this setup, at each iteration \(j\ge 0\), the system applies the logistic
map \(f_{c(\theta^{j}\omega)}\) corresponding to the current realisation of
the parameter.

	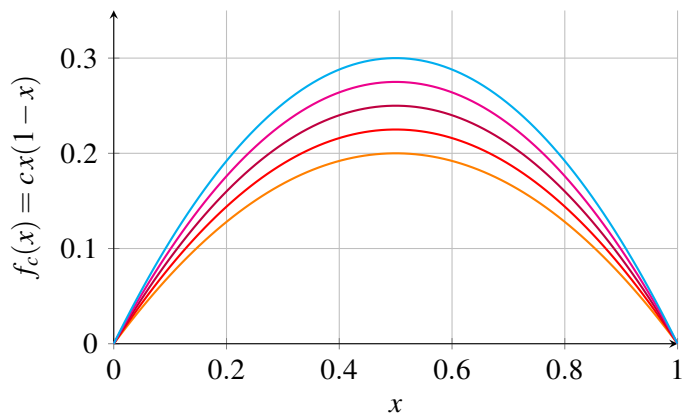
\begin{figure}[ht]
  \centering
  \begin{tikzpicture}[scale=1]
    \begin{axis}[
      width=9cm, height=6cm,
      xlabel={$x$},
      ylabel={$f_c(x) = c\,x(1-x)$},
      axis lines=left,
      grid=both,
      xmin=0, xmax=1,
      ymin=0, ymax=0.35
    ]
      \addplot[thick, orange]  expression[domain=0:1, samples=100]{0.8 * x * (1 - x)};
      \addplot[thick, red]     expression[domain=0:1, samples=100]{0.9 * x * (1 - x)};
      \addplot[thick, purple]  expression[domain=0:1, samples=100]{1.0 * x * (1 - x)};
      \addplot[thick, magenta] expression[domain=0:1, samples=100]{1.1 * x * (1 - x)};
      \addplot[thick, cyan]    expression[domain=0:1, samples=100]{1.2 * x * (1 - x)};
    \end{axis}
  \end{tikzpicture}
  \caption{Graphs of the logistic maps $f_c(x)=c\,x(1-x)$ on $[0,1]$ for 
  the parameter values $c=0.8,0.9,1.0,1.1$ and $1.2$.}
  \label{fig:logistic-small-c}
\end{figure}

To apply our main result, fix a distinguished parameter \(c_0 \in (0,4]\) such
that the deterministic map \(f_{c_0}\) possesses a hyperbolic periodic orbit.
Such orbits can be identified, for example, by using the topological
conjugacy between the logistic map with \(c_0=4\) and the tent map, or by
choosing \(c_0\) in a periodic window. Let \(N\subset[0,1]\) be an isolating
neighbourhood for this periodic orbit under \(f_{c_0}\). By the continuity of
the family \(c\mapsto f_c\), there exists a neighbourhood \(U\) of \(c_0\)
such that \(N\) remains an isolating neighbourhood for \(f_c\) for all
\(c \in U\).

Assuming that the support \(C\) of the random variable \(c(\cdot)\) is
contained in \(U\) (that is, the random perturbations of the parameter are
sufficiently small), the set \(N\) defines a random isolating neighbourhood
for the RDS \(\varphi\). Invoking the continuation property of the random
Conley index (Theorem~\ref{ContinuationProperty}), the index of the maximal
invariant set inside \(N\) for the random system coincides with that of the
deterministic periodic orbit for \(f_{c_0}\). Consequently, the random system
\(\varphi\) admits a non-trivial random invariant set within \(N\) that
inherits the topological signature—and hence the minimal period—of the
original deterministic orbit.

Theorem~\ref{thm:main_forcing} then ensures that this deterministic period is
realised as a random minimal period for sufficiently small random
perturbations of the logistic map.

\section*{Acknowledgements}
The authors would like to thank Mariana Rodrigues da Silveira for many insightful discussions and, in particular, for her substantial contributions to our understanding of the deterministic Conley index. Isabella Alvarenga gratefully acknowledges financial support from CNPq (Brazil) for her Master’s program, during which part of this research was carried out.

\bibliographystyle{plain}

\bibliography{references.bib}

@book{willard2012general,
  title={General topology},
  author={Willard, Stephen},
  year={2012},
  publisher={Courier Corporation},
  address   = {Mineola}
}

@book{arnold1995random,
  title     = {Random Dynamical Systems},
  author    = {Arnold, Ludwig},
  year      = {1998},
  publisher = {Springer},
  address   = {Berlin}
}

@article{liu2008conley,
  title={Conley index for random dynamical systems},
  author={Liu, Zhenxin},
  journal={Journal of Differential Equations},
  volume={244},
  number={7},
  pages={1603--1628},
  year={2008},
  publisher={Elsevier}
}

@article{szymczak1995conley,
  title={The Conley index for discrete semidynamical systems},
  author={Szymczak, Andrzej},
  journal={Topology and its Applications},
  volume={66},
  number={3},
  pages={215--240},
  year={1995},
  publisher={Elsevier}
}

@article{andres2008randomization,
  title={Randomization of Sharkovskii-type theorems},
  author={Andres, Jan},
  journal={Proceedings of the American Mathematical Society},
  volume={136},
  number={4},
  pages={1385--1395},
  year={2008}
}

@article{klunger2001periodicity,
  title={Periodicity and Sharkovsky's theorem for random dynamical systems},
  author={Kl{\"u}nger, Marc},
  journal={Stochastics and Dynamics},
  volume={1},
  number={03},
  pages={299--338},
  year={2001},
  publisher={World Scientific}
}

@article{sharkovskiui1995coexistence,
  title={Coexistence of cycles of a continuous map of the line into itself},
  author={Sharkovskii, A. N.},
  journal={International journal of bifurcation and chaos},
  volume={5},
  number={05},
  pages={1263--1273},
  year={1995},
  publisher={World Scientific}
}

@article{Barbarski2011,
  title        = {The Sharkovskii theorem for spaces of measurable functions},
  author       = {Barbarski, Pawel},
  journal      = {Journal of Mathematical Analysis and Applications},
  volume       = {373},
  number       = {2},
  pages        = {414--421},
  year         = {2011},
  publisher    = {Elsevier}
}

@article{AndresBarbarski2016,
  title   = {Randomized Sharkovsky-type results and random subharmonic solutions of differential inclusions},
  author  = {Andres, Jan and Barbarski, Pawel},
  journal = {Proceedings of the American Mathematical Society},
  volume  = {144},
  number  = {5},
  pages   = {1971--1983},
  year    = {2016}
}

@article{LiYorke1975,
  title        = {Period three implies chaos},
  author       = {Li, Tien-Yien and Yorke, James A.},
  journal      = {The American Mathematical Monthly},
  volume       = {82},
  number       = {10},
  pages        = {985--992},
  year         = {1975}
}

@article{Kloeden1979,
  title   = {On Sharkovsky's cycle coexistence ordering},
  author  = {Kloeden, Peter E.},
  journal = {Bulletin of the Australian Mathematical Society},
  volume  = {21},
  number  = {2},
  pages   = {185--192},
  year    = {1980}
}

@article{BurnsHasselblatt,
  author    = {Burns, Keith and Hasselblatt, Boris},
  title     = {The {Sharkovsky} Theorem: A Natural Direct Proof},
  journal   = {The American Mathematical Monthly},
  volume    = {118},
  number    = {3},
  pages     = {229--244},
  year      = {2011},
  publisher = {Taylor \& Francis},
}

@article{Baldwin1991,
  title   = {An extension of Sharkovsky's theorem to the $n$-od},
  author  = {Baldwin, Scott},
  journal = {Ergodic Theory and Dynamical Systems},
  volume  = {11},
  pages   = {249--271},
  year    = {1991}
}

@book{AlsedaLlibreMisiurewicz2000,
  title        = {Combinatorial dynamics and entropy in dimension one},
  author       = {Alsed{\`a}, Llu{\'\i}s and Llibre, Jaume and Misiurewicz, Micha{\l}},
  publisher    = {World Scientific},
  year         = {2000},
  address   = {Singapore},

}

@article{Andres2019,
  title={Application of the randomized Sharkovsky-type theorems to random impulsive differential equations and inclusions},
  author={Andres, Jan},
  journal={Journal of Dynamics and Differential Equations},
  volume={31},
  number={4},
  pages={2127--2144},
  year={2019},
  publisher={Springer}
}

@phdthesis{villapouca2013teoria,
  title={A teoria do {\'\i}ndice de Conley discreta para conjuntos b{\'a}sicos zero-dimensionais},
  author={Villapouca, Mariana Gesualdi},
  year={2013},
  school = {Universidade Estadual de Campinas},
address = {Campinas}
}

@article{mrozek1990leray,
  title={Leray functor and cohomological Conley index for discrete dynamical systems},
  author={Mrozek, Marian},
  journal={Transactions of the American Mathematical Society},
  volume={318},
  number={1},
  pages={149--178},
  year={1990}
}

@article{franks2000shift,
  title={Shift equivalence and the Conley index},
  author={Franks, John and Richeson, David},
  journal={Transactions of the American Mathematical Society},
  volume={352},
  number={7},
  pages={3305--3322},
  year={2000}
}

\end{document}